\documentclass[10pt, reqno]{amsart}
\usepackage{amssymb,amsfonts,lineno,amsthm,amscd,stmaryrd}
\usepackage[leqno]{amsmath}
\usepackage[active]{srcltx}
\usepackage[colorlinks,pdfpagelabels,pdfstartview = FitH,bookmarksopen = true,bookmarksnumbered = true,linkcolor = blue,plainpages = false,hypertexnames = false,citecolor = red]{hyperref}

\numberwithin{equation}{section} 

\newtheorem{theorem}{Theorem}[section]
\newtheorem{lemma}[theorem]{Lemma}
\newtheorem{proposition}[theorem]{Proposition}

\theoremstyle{definition}
\newtheorem{definition}[equation]{Definition}

\theoremstyle{remark}
\newtheorem{remark}[equation]{Remark}

\def\dx{\,dx}

\def\dt{\,dt}
\def\ds{\,ds}

\newcommand{\om}{{\omega(\cdot)}}

\newcommand{\R}{\mathbb{R}}
\newcommand{\N}{\mathbb{N}}

\DeclareMathOperator{\divergence}{div}
\DeclareMathOperator{\loc}{loc}
\DeclareMathOperator*{\osc}{osc}

\newcommand{\intav}[1]{\mathchoice {\mathop{\vrule width 6pt height 3 pt depth  -2.5pt
\kern -8pt \intop}\nolimits_{\kern -6pt#1}} {\mathop{\vrule width
5pt height 3  pt depth -2.6pt \kern -6pt \intop}\nolimits_{#1}}
{\mathop{\vrule width 5pt height 3 pt depth -2.6pt \kern -6pt
\intop}\nolimits_{#1}} {\mathop{\vrule width 5pt height 3 pt depth
-2.6pt \kern -6pt \intop}\nolimits_{#1}}}

\def\mean#1{\mathchoice%
          {\mathop{\kern 0.2em\vrule width 0.6em height 0.69678ex depth -0.58065ex
                  \kern -0.8em \intop}\nolimits_{\kern -0.4em#1}}%
          {\mathop{\kern 0.1em\vrule width 0.5em height 0.69678ex depth -0.60387ex
                  \kern -0.6em \intop}\nolimits_{#1}}%
          {\mathop{\kern 0.1em\vrule width 0.5em height 0.69678ex
              depth -0.60387ex
                  \kern -0.6em \intop}\nolimits_{#1}}%
          {\mathop{\kern 0.1em\vrule width 0.5em height 0.69678ex depth -0.60387ex
                  \kern -0.6em \intop}\nolimits_{#1}}}

\def\vintslides_#1{\mathchoice%
          {\mathop{\kern 0.1em\vrule width 0.5em height 0.697ex depth -0.581ex
                  \kern -0.6em \intop}\nolimits_{\kern -0.4em#1}}%
          {\mathop{\kern 0.1em\vrule width 0.3em height 0.697ex depth -0.604ex
                  \kern -0.4em \intop}\nolimits_{#1}}%
          {\mathop{\kern 0.1em\vrule width 0.3em height 0.697ex depth -0.604ex
                  \kern -0.4em \intop}\nolimits_{#1}}%
          {\mathop{\kern 0.1em\vrule width 0.3em height 0.697ex depth -0.604ex
                  \kern -0.4em \intop}\nolimits_{#1}}}

\newcommand{\aveint}[2]{\mathchoice%
          {\mathop{\kern 0.2em\vrule width 0.6em height 0.69678ex depth -0.58065ex
                  \kern -0.8em \intop}\nolimits_{\kern -0.45em#1}^{#2}}%
          {\mathop{\kern 0.1em\vrule width 0.5em height 0.69678ex depth -0.60387ex
                  \kern -0.6em \intop}\nolimits_{#1}^{#2}}%
          {\mathop{\kern 0.1em\vrule width 0.5em height 0.69678ex depth -0.60387ex
                  \kern -0.6em \intop}\nolimits_{#1}^{#2}}%
          {\mathop{\kern 0.1em\vrule width 0.5em height 0.69678ex depth -0.60387ex
                  \kern -0.6em \intop}\nolimits_{#1}^{#2}}}
                  
\newcommand{\eps}{\varepsilon}

\begin{document}

\title[Quantitative modulus of continuity]{A quantitative modulus of continuity\\ for the two-phase Stefan problem}

\author{Paolo Baroni}
\address{Uppsala University, Department of Mathematics, L\"agerhyddsv\"agen 1, SE-751 06 Uppsala, Sweden.}
\email{paolo.baroni@math.uu.se}

\author{Tuomo Kuusi}
\address{Aalto University, Institute of Mathematics, P.O. Box 11100, FI-00076 Aalto, Finland.}
\email{tuomo.kuusi@tkk.fi}

\author{Jos\'e Miguel Urbano}
\address{CMUC, Department of Mathematics, University of Coimbra, 3001-501 Coimbra, Portugal.}
\email{jmurb@mat.uc.pt}

\allowdisplaybreaks
\date{\today}


\begin{abstract}
We derive the quantitative modulus of continuity
$$
\omega(r)=\left[  p+\ln \left( \frac{r_0}{r} \right) \right]^{-\alpha (n,p)},
$$
which we conjecture to be optimal, for solutions of the $p$-degenerate two-phase Stefan problem.
Even in the classical case $p=2$, this represents a twofold improvement with respect to the 1984 state-of-the-art result by DiBenedetto and Friedman \cite{DiBeFrie84}, in the sense that we discard one logarithm iteration and obtain an explicit value for the exponent $\alpha (n,p)$.
\end{abstract}

\maketitle

\section{Introduction}

This paper concerns  the local behaviour of bounded weak solutions of the degenerate two-phase Stefan problem
\begin{equation} \label{the equation}
\partial_t \big[u+{\mathcal L}_h \, H_a(u)\big] \ni \mathrm{div}\, \big[|Du|^{p-2} Du\big]\,, \qquad p\geq 2\,,
\end{equation}
where $H_a$ is the Heaviside graph centred at $a\in\R$, defined by
\begin{equation}\label{Heavi}
H_a(s)=
\begin{cases}
0 & \text{if $s<a$}\,,\\[1pt]
[0,1] & \text{if $s=a$}\,,\\[1pt]
1 & \text{if $s> a$}\,,
 \end{cases}
\end{equation}
and ${\mathcal L}_h>0$. Our main result is the derivation of the explicit, interior modulus of continuity 
\begin{equation}\label{modulus_of_continuity}
\omega (r) := \left[  p+\ln \left( \frac{r_0}{r} \right) \right]^{-\alpha (n,p)},\qquad 0<r\leq r_0\,,
\end{equation}
which we conjecture to be optimal.

An extensive literature, both from the theoretical and the computational points of view, is available for the classical Stefan problem
\begin{equation}\label{Stefan}
\partial_t \big[u+{\mathcal L}_h \, H_0(u)\big] \ni \triangle u,
\end{equation}
corresponding to the case $p=2$, which is a simplified model to describe the evolution of the configuration of a substance which is changing phase, when convective effects are neglected. The function $u$ represents the temperature and the value $u=0$ is the level at which the change of phase occurs; the height ${\mathcal L}_h$ of the jump of the graph ${\mathcal L}_h \, H_0(\cdot)$ corresponds to the latent heat of fusion and a selection of the graph is called the {\em enthalpy} of the problem. For simplicity, we consider ${\mathcal L}_h \leq 1$ from now on. The case of study of positive solutions (note we are taking $a=0$ in \eqref{Stefan}) is usually called one-phase Stefan problem, while if no sign assumptions are made on $u$ we are dealing with the    two-phase Stefan problem; see \cite{Friedman, russ} for the deduction of \eqref{Stefan} from the classic formulation, which goes back to Stefan at the end of the nineteenth century \cite{Ste_before_history} and has been subsequently developed in \cite{Kam, Ole}. We mention that the model \eqref{Stefan} also finds applications in finance \cite{finance_book}, biology related to the Lotka-Volterra model \cite{Bio1}, and flows of solutes or gases in porous media \cite{Rubin}.

Clearly \eqref{Stefan} and \eqref{the equation} need to be understood using an appropriate notion of (weak) solution, and the one we employ is that of {\em differential inclusion} in the sense of graphs, see Definition \ref{deff}; other approaches can be used and the most noticeable one is that of {\em viscosity solutions} in the sense introduced by Crandall and Lions, and developed by Caffarelli; see \cite{ACS1} and the recent survey by Salsa \cite{Salsa12}. Notice that weak solutions are viscosity solutions once one knows they are continuous (and in fact they are, see the following lines); under an additional conditions (namely, $\{u=0\}$ is negligible) the converse also holds true, see \cite{Kim}.

For the one-phase Stefan problem \eqref{Stefan}, continuity of weak solutions has been proved by Caffarelli and Friedman in \cite{CaFri79}, with an explicit modulus of continuity:
\[
C\left[ \ln \left( \frac{r_0}{r} \right)\right]^{-\epsilon},  \quad \mathrm{if} \ n\geq3\,; \qquad C\, 2^{-\left[ \ln \left( \frac{r_0}{r} \right)\right]   ^\gamma},   \quad \mathrm{if} \  n=2\,,
\]
for a positive constant $C$, for any $0<\epsilon < \frac{2}{n-2}$ and for any $0<\gamma<\frac12$. For the two-phase problem, continuity was proved, almost at the same time, by Caffarelli and Evans \cite{CafEv83} for \eqref{Stefan}, and by DiBenedetto \cite{DiB2}, who considered more general, nonlinear structures for the elliptic part, albeit with linear growth with respect to the gradient, and lower order terms depending on the temperature, which is relevant when convection is taken into  account: 
\begin{equation}\label{general_Stefan}
\partial_t \big[u+H_0(u)\big]\ni\divergence a(x,t,u,Du)+b(x,t,u,Du),\quad a(x,t,u,Du)\approx Du; 
\end{equation}
see also \cite{Ziemer, Sacks}. More general structures including multi-phase Stefan problem were considered in~\cite{DiBeVesp95}.  Moreover, in the first of their celebrated papers about the gradient regularity for solutions of parabolic $p$-Laplace equations, DiBenedetto and Friedman state (without proof) that the method of the paper yields as modulus of continuity for the solutions of the two-phase Stefan problem
\begin{equation}\label{loglog}
\omega(r)=\left[  \ln \ln \left( \frac{Ar_0}{r} \right) \right]^{-\sigma}, \quad \text{for some $A,\sigma>0$\,,} 
\end{equation}
see \cite[Remark 3.1]{DiBeFrie84}; this seems to be the first instance in which an explicit modulus of continuity appears in the study of the two-phase Stefan problem. Details are somehow pointed out in \cite{DiB3}, where DiBenedetto shows that, in the case of H\"older continuous boundary data, the solution of the Cauchy-Dirichlet problem for equation \eqref{general_Stefan} has modulus of continuity \eqref{loglog}, giving a quantitative form to the up-to-the-boundary continuity result previously proved by Ziemer in \cite{Ziemer}. These, to the best of our knowledge, are the last quantitative results concerning the continuity of the solutions of the classical two-phase Stefan problem.

For the degenerate case, $p>2$ in \eqref{the equation}, very little is known. Existence was obtained  by one of the authors in \cite{UEx} using an approximation method; subsequently he proved the continuity \cite{UCon} of at least one of them, in the spirit of \cite{DiB2}, circumventing the additional difficulties resulting from the presence of the $p$-Laplacian in the elliptic part. The continuity proof only leads to an implicit modulus of continuity.

\vspace{3mm}

Our derivation of the modulus of continuity \eqref{modulus_of_continuity} represents an improvement with respect to the state-of-the-art in several ways: we discard an iteration of the logarithm, reaching what we conjecture to be the sharp, optimal modulus of continuity for the two-phase Stefan problem; we determine the precise value of the exponent $\alpha$ in terms of the data of the problem; we cover the degenerate case $p > 2$ and we provide a comprehensive proof.

\subsection{Statement of the problem and main result}

More generally, we shall consider the following extension of~\eqref{the equation}:
\begin{equation} \label{the general equation}
\partial_t \big[\beta(u)+{\mathcal L}_h \, H_a(\beta(u))\big] \ni \mathrm{div}\, \mathcal{A}(x,t,u,Du) \qquad\text{ in $\;\Omega_T:=\Omega\times(0,T)$}\,,
\end{equation}
where $\Omega$ is a bounded domain of $\R^n$, $n\geq2$. $H_a$ is defined in \eqref{Heavi}, $\beta:\R\to\R$ is a $C^1$-diffeomorphism such that $\beta(0)=0$ and satisfying the  bi-Lipschitz condition
\begin{equation*}
\Lambda^{-1} |u-v| \leq |\beta(u) - \beta(v)| \leq \Lambda |u-v|
\end{equation*}
for some given $\Lambda\geq1$ and the vector field $\mathcal{A}$ is measurable with respect to the first two variables and continuous
with respect to the last two, satisfying, in addition, the following growth and ellipticity assumptions:
\begin{equation}\label{p.laplacian.assumptions}
\qquad|\mathcal{A}(x,t,u,\xi)| \leq \Lambda |\xi|^{p-1} \,, \qquad \langle \mathcal{A}(x,t,u,\xi) , \xi \rangle \geq \Lambda^{-1} |\xi|^p \,;
\end{equation}
the previous inequalities are intended to hold for almost any $(x,t)\in\Omega_T$ and for all  $(u,\xi)\in\R\times\R^n$. We consider the bi-Lipschitz function $\beta$  in order to include thermal properties of the medium, which may slightly change with respect to the temperature, as already done in \cite{DiB2, Noche87}.

\begin{definition}\label{deff}
A {\em local weak solution} to equation \eqref{the general equation} is a function
\[
u\in  L^\infty_{\loc}(0,T;L^2_{\loc}(\Omega))\cap L^p_{\loc}(0,T;W^{1,p}_{\loc}(\Omega))=:V^{2,p}_{\loc}(\Omega_T)
\]
such that a selection $v\in \beta(u)+{\mathcal L}_h H_a(\beta(u))$ satisfies the integral identity
\[
\int_{\mathcal K}[ v\,\varphi](\cdot,\tau) \dx\biggr|_{\tau=t_1}^{t_2} +\int_{\mathcal K\times[t_1,t_2]} \big[-v\, \partial_t\varphi  +\langle \mathcal{A}(\cdot, \cdot, u,Du) ,D\varphi\rangle\big]\dx\dt= 0
\]
for all $\mathcal K\Subset\Omega$, almost every $t_1,t_2\in\R$ such that $[t_1,t_2]\Subset(0,T)$ and for every test function $\varphi \in L^p_{\loc}(0,T;W^{1,p}_0(\mathcal K))$ such that $\partial_t\varphi\in L^2(\mathcal K\times[t_1,t_2])$.
\end{definition}

We assume in this paper that a {\em local weak solution} can be obtained as a locally uniform limit of locally H\"older continuous solutions to \eqref{the general equation} for a regularized graph, see Section~\ref{sec: approx}. In \cite{BKU2} we construct such a solution for the Cauchy-Dirichlet problem with continuous boundary datum; we derive, in addition, an explicit modulus of continuity up to the boundary. We refer to \cite{Friedman_book, russ} for the existence of weak solutions for bounded Cauchy-Dirichlet data in the case $p=2$. For the case $p>2$ the solution, whose existence can be retrieved from \cite{UCon}, is known to be unique just for homogeneous Dirichlet data, see \cite{NoUr03}.

Our main result is the derivation of a quantitative modulus of continuity for local weak solutions to \eqref{the general equation}.
\begin{theorem}\label{interior.continuity}
Let $u$ be a local weak solution to \eqref{the general equation}, obtained by approximation, and let 
\begin{equation}\label{ratio.kappa}
\alpha :=  \begin{cases}
  \displaystyle{\frac{p}{n+p}} \qquad&  \text{for }p<n\,,\\[3pt]
  \mbox{any number $\displaystyle{<\frac12}$} \qquad& \text{for }p=n\,,\\[5pt]
   \displaystyle{\frac12}\qquad & \text{for }p>n\,.
\end{cases}
 \end{equation}
Then there exist constants $M,L$ and $c_\ast$, larger than one and depending only on $n,p,\Lambda$ and $\alpha$, such that, considering the modulus of continuity
\begin{equation} \label{choice modulus}
\omega (r) =  L\left[ p+{\ln\Bigl( \frac{r_0}{r}\Bigr)}\right]^{-\alpha}
\end{equation}
and cylinders
\begin{equation}\label{cyl}
\overline Q_r^{\omega(\cdot)} :=B_{r}(x_0)\times (t_0-M \max\{\osc_{\Omega_T} u,1 \}^{2-p} [ \omega(r) ]^{(2-p)(1+1/\alpha)} r^p,t_0)\,,
\end{equation}
we have that if $\overline Q_{r_0}^{\omega(\cdot)}\subset\Omega_T$, then
\begin{equation}\label{oscillation.interior}
 \osc_{\overline Q_r^{\omega(\cdot)}} u \leq c_\ast \, \omega(r) \max\{\osc_{\Omega_T} u,1 \} 
\end{equation}
holds for all $r \in (0,r_0]$.
\end{theorem}

\begin{remark}
By choosing $M$ large enough, it is rather straightforward to see that the above defined space-time cylinders 
$\overline  Q_r^{\omega(\cdot)} $ satisfy
\[
\overline  Q_{r_1}^{\omega(\cdot)}\subset \overline  Q_{r_2}^{\omega(\cdot)}\subset \overline  Q_{r_0}^{\omega(\cdot)}
\]
whenever $0<r_1\leq r_2\leq r_0$. Moreover, $r\mapsto \omega(r)$ is concave for $0<r\leq r_0$. For details, see Section~\ref{deriving}.
\end{remark}

\begin{remark}\label{blup}
{\rm Observe that, in the above theorem, $\alpha$ can be taken arbitrarily close to $1/2$ in the case $p=n$; however, the constants $c_\ast$ and $M$ in Theorem~\ref{interior.continuity} blow up as $\alpha\uparrow1/2$.}
\end{remark}

\subsection{Some notes about the proof.} We explain here, briefly and formally, the main ideas behind the continuity proofs, which can perhaps be blurred by the technical details. 

We shall work with approximate solutions $u_\eps$ and show ultimately that
\[
 \osc_{\overline Q_r^{\omega(\cdot)}} u_\eps \leq c_\ast\, \omega(r) \max\{\osc_{\Omega_T} u,1 \} +c\,\eps\,,
 \]
where  $\omega(\cdot)$ and $\overline Q_r^{\omega(\cdot)}$ are defined in \eqref{choice modulus}-\eqref{cyl}, $u_\eps$ is the solution to the approximating equation \eqref{eq.app} and $c$ does not depend on $\eps$. From this it will be easy to deduce Theorem \ref{interior.continuity} simply by taking the limit as $\eps\downarrow0$ and using the convergence of $u_\eps$ to $u$.

After fixing a cylinder $Q\equiv \overline Q_{r_0}^{\omega(\cdot)}$ as above, we can suppose, up to translation and rescaling, that $\sup u=\osc u \leq 1$ on $Q$ (we are omitting the $\eps$ for simplicity). Moreover, we can clearly suppose that $\osc u>\omega(r)$ and also that the jump is in the interval $[\osc u/2, \osc u]$ (note that if the jump is outside $[0,\osc u]$ there is nothing to prove, since we are dealing with the parabolic $p$-Laplace equation in $Q$), see subsection \ref{reductions}. Next we fix a  classical alternative: either $\sup u=\osc u$ is greater than $\omega(r)/4$ in a large portion of the cylinder $\widetilde Q_r^\om\subset Q$ \eqref{alt.1}, or this does not hold \eqref{alt.2}. Here, $\widetilde Q_r^\om$ is an appropriate cylinder, whose time-scale differs from that used for $\overline Q_r^{\omega(\cdot)}$. 

In the case that \eqref{alt.1} holds true, we truncate the solution below the jump, obtaining a weak supersolution to the parabolic $p$-Laplace equation, and we use the weak Harnack inequality, together with \eqref{alt.1} to lift up the infimum of $u$, therefore reducing the oscillation. Note that here we shall use  that the jump belongs to the interval $[\osc u/2, \osc u]$ in order to have enough room to make the truncation possible.

In the second case, we use Caccioppoli's inequality to perform a De Giorgi iteration, starting from \eqref{alt.2}. We have to use two tools in order to rebalance the high degeneracy of the problem, caused both by the jump (which produces an $L^1$ term on the right-hand side of the energy estimate, see \eqref{caccioppoli estimate}) and by the degeneracy of the $p$-Laplacian: the latter is rebalanced by the size of the cylinder, which depends on $\omega(r)$, see $\widetilde T_r^\om$ in  \eqref{choices T_r and T_r^1}, while the former is rebalanced by the fact that we introduce $\omega(r)$ in the size conditions of the alternatives, see again \eqref{alt.1}-\eqref{alt.2}. Notice that, in the case $p=2$, the cylinders we consider are the standard parabolic ones, $B_r(x_0)\times(t_0-M\,r^2,t_0)$, for a large but universal constant $M$; hence, for the logarithmic continuity for the classical Stefan problem \eqref{Stefan}, the trick essentially consists in rebalancing the presence of the jump with an alternative involving the modulus of continuity itself. Having reduced the supremum of $u$ on a part of the cylinder (see \eqref{friday}), using the time scale given by $\widetilde T_r^\om$, we forward this information in time using a logarithmic estimate and then perform another De Giorgi iteration, this time using the second time scale $T_r^\om$ in \eqref{choices T_r and T_r^1} to rebalance the eventual degeneracy due to the $p$-Laplacian operator. Considering the two alternatives, we see that the choice of the modulus of continuity \eqref{modulus_of_continuity} is the correct one, allowing to merge the two different options and to make the iteration scheme work, see Section \ref{deriving}. 

Finally, we would like to highlight the points of contact of our paper with the recent work \cite{KMN}, where sharp continuity results are proved for obstacle problems involving the evolutionary $p$-Laplacian operator. There, it is shown that, once considering obstacles with modulus of continuity $\omega(\cdot)$ (where here this expression has to be understood in an appropriate, intrinsic way), the solution has the same regularity, in the sense that it has the same modulus of continuity. In order to get such result, the authors have to deal with particular cylinders of the form (take as the center the origin, for simplicity)
\[
Q^{\lambda\omega(\cdot)}_r:= B_r \times ( - [\lambda \omega(r)]^{2-p}r^p, 0), \quad\text{ with }\quad \lambda\approx\frac{\osc_{Q^{\lambda\omega(\cdot)}_r}u}{\omega(r)}
\]
and where $u$ is the solution they are considering; these cylinders are the ones involved also in the intrinsic definition of the modulus of continuity and they allow to rebalance the inhomogeneity of the problem. This is an extension of DiBenedetto's approach to regularity for the parabolic $p$-Laplacian, see~\cite{AM07, DiBe93, Urba08}, where results are recovered as extremal cases of a family of general interpolative intrinsic geometries. Notice the similitude with the cylinders defined in \eqref{cyl} and the fact that we also have to deal with the further inhomogeneity given by the jump; this precisely reflects in the  presence of the exponent $1+1/\alpha$. 

\subsection{Notation}
Our notation will be mostly self-explanatory; we mention here some noticeable facts. We shall follow the usual convention of denoting by $c$ a generic constant, always greater or equal than one, that may vary from line to line; constants we shall need to recall will be denoted with special symbols, such as $\tilde c, c_\ast,c_1$ or the like. Dependencies of constants will be emphasised between parentheses: $c(n,p,\Lambda)$ will mean that $c$ depends only on $n,p,\Lambda$;  they will often be indicated just after displays. The dependence of constants upon $\alpha$ (and on $\kappa$, see \eqref{eq:kappa}) will be meaningful only in the case $p=n$; in the case $p<n$ this would just add a dependence on $n,p$ -- see also Remark \ref{blup}. Unless otherwise stated, we shall avoid to indicate the centre of the ball when it will be the zero vector: $B_r:=B_r(0)$.

Being $A \in \R^k$ a measurable set with positive measure and $f:A \to \R^m$ an integrable map,  with $k,m \ge 1$, we shall denote with ${(f)}_A$ the averaged integral
\begin{equation*}
      {(f)}_{A}:=\mean{A} f(\xi)\, d\xi := \frac{1}{|A|} \int_{A} f(\xi)\, d\xi	\,.
\end{equation*}
We stress that with the statement ``a vector field with the same structure as $\mathcal A$'' (or ``structurally similar to $\mathcal A$'', or similar expressions) we shall mean that the vector field $\mathcal A$ satisfies \eqref{p.laplacian.assumptions}, possibly with $\Lambda$ replaced by a constant depending only on $n,p$ and $\Lambda$, and continuous with respect to the last two variables. 

Finally, by $\ln\ln x$, for $x>1$, we will mean $\ln(\ln x)$; $\N$ will be the set $\{1,2,\dots\}$, while $\N_0:=\N\cup\{0\}$; $\R^+:=[0,\infty)$.

\section{Collecting tools}

\subsection{Approximation of the problem}\label{sec: approx}
Let $\rho_\eps$ be the standard symmetric, positive one dimensional mollifier supported in $(-\eps,\eps)$. Set
\begin{equation*}
H_{a,\eps}(s) := (\rho_\eps \ast H_a)(s)\qquad\text{for $s\in\R$}\,;
\end{equation*}
then $H_{a,\eps}$ is smooth. Moreover, the support of $H_{a,\eps}'$ is contained in $(a-\eps,a+\eps)$. 
Let $\{u_\eps\}$ be a sequence converging locally uniformly to $u$ as $\eps\downarrow0$, where $u_\eps$ is a weak solution to the approximate equation
\begin{equation}\label{eq.app}
\partial_t \big[\beta(u_\eps)+ {\mathcal L}_h\, H_{a,\eps}(\beta (u_\eps))\big] - \mathrm{div}\, \mathcal{A}(x,t,u_\eps,Du_\eps) = 0\qquad\text{in $\Omega_T$}.  
\end{equation}
Now, setting
\begin{equation}\label{w}
w := \beta(u_\eps)\,, 
\end{equation}
we arrive at the  regularized equation
\begin{equation} \label{eq:mollified eq}
\partial_t w -  \mathrm{div}\, \widetilde{\mathcal{A}}(x,t,w,Dw) = -{\mathcal L}_h\, \partial_t H_{a,\eps}(w)\,,
\end{equation}
where
\begin{equation*} 
\widetilde{\mathcal{A}}(x,t,w,Dw):= \mathcal{A}(x,t,\beta^{-1}(w), [\beta^\prime(\beta^{-1}(w))]^{-1} Dw)\,.
\end{equation*}
Observe that the growth and ellipticity bounds for $\widetilde{\mathcal{A}}$ are inherited from $\mathcal{A}$ and from the two-sided bound for 
$\beta'$: indeed, we in particular get that
\begin{equation}\label{tilde inf}
|\widetilde{ \mathcal{A}}(x,t,u,\xi)| \leq \Lambda^p |\xi|^{p-1} \,, \qquad \langle \widetilde{\mathcal{A}}(x,t,u,\xi) , \xi \rangle \geq \Lambda^{-p}
|\xi|^p
\end{equation}
for almost every $(x,t) \in \Omega_T$ and for all $(u,\xi) \in \mathbb{R} \times  \mathbb{R}^n$. Moreover, 
$\widetilde{\mathcal{A}}$ is clearly continuous with respect to the last two variables since $\beta$ is $C^1$-diffeomorphism. Note that we dropped $\eps$ from the notation; it will be recovered in \eqref{recover epsilon}.

By regularity theory for evolutionary $p$-Laplace type equations, see~\cite{DiBe93, Urba08}, we actually have that the solution $w$ is H\"older 
continuous since $\beta(u_\eps)+{\mathcal L}_h H_{a,\eps}(\beta (u_\eps))$ is a diffeomorphism. However, this kind of regularity depends on the regularization and, in 
particular, will deteriorate as $\eps \downarrow 0$. Nonetheless, we may assume that the solution  $w$ to the regularized equation is continuous having pointwise values.

\subsection{Scaling of the equation}\label{scaling}
Once given a function $z$ solving \eqref{eq.app} or \eqref{eq:mollified eq} in $B_r(x_0)\times(t_0- \lambda^{2-p}\overline T,t_0)$, for some $\overline T>0,\lambda\geq1$, if we consider the function 
\[
\bar z(y,s) := \lambda^{-1} z(x_0+y,t_0- \lambda^{2-p}(T_0+\overline T)+ \lambda^{2-p} s),\quad (y,s)\in B_r\times(T_0,T_0+\overline T)\,,
\]
it is easy to see that $\bar z$ solves an equation which is structurally similar to the one solved by $z$, but with a multiplier $\lambda^{-1} {\mathcal L}_h \in [0,1]$ for the phase-transition term. 

\subsection{Space-time geometry}\label{spgeo}
We set
\begin{equation} \label{choices T_r and T_r^1}
\widetilde T_r^{\omega(\cdot)} := [ \omega(r) ]^{2-p} r^p \qquad \mbox{and} \qquad T_r^{\omega(\cdot)} := M[ \omega(r) ]^{(2-p)(1+1/\alpha)} r^p \,,
\end{equation}
for the modulus of continuity $\omega$ defined in \eqref{choice modulus}, with $L\geq1$ and $M\geq2$ to be fixed; we also set 
\begin{equation*}
\widetilde Q_r^\om = B_{r/4} \times \big(0,\widetilde T_r^\om\big)\qquad\text{and}\qquad Q_r^{\omega(\cdot)}:=B_r\times \big(0,T_r^{\om}\big)\,.
\end{equation*}
We stress that up to Section \ref{deriving} it will be sufficient to think of $\omega$ simply as a generic concave modulus of continuity, such that the maps $r\mapsto T_r^{\omega(\cdot)}$ and $r\mapsto \widetilde T_r^{\omega(\cdot)}$ are monotone increasing, i.e., 
\[
0<r_1\leq r_2\qquad\Longleftrightarrow\qquad 0<T_{r_1}^\om\leq T_{r_2}^\om\quad\text{and}\quad0<\widetilde{T}_{r_1}^\om\leq \widetilde{T}_{r_2}^\om\,.
\]
Notice, on the other hand, that we still have not chosen the value of $L$. The explicit expression in \eqref{choice modulus} will indeed be used only in Section \ref{deriving} when iterating the reduction of oscillation obtained in the forthcoming Section \ref{reducing}, in order to obtain \eqref{oscillation.interior}. Needless to say, the choice in \eqref{Elle} will imply that time scales are monotone, in the above sense.

\subsection{Energy estimates}
We consider in this subsection continuous weak solutions to the following equation
\begin{equation}\label{eq.prova}
\partial_t v - \mathrm{div} \, \widetilde{\mathcal{A}}(x,t,v,Dv)  = - \widetilde{{\mathcal L}_h} \, \partial_t H_{b,\eps}(v)\,,
\end{equation}
where $\widetilde{\mathcal A}$ has the same structure of $\mathcal A$, $b\in\R$, and $\widetilde{{\mathcal L}_h}  \in [0,1]$; we shall, in particular, use the next results for equation \eqref{the equation for v}, with $b$ defined in \eqref{v}. The following is a Caccioppoli's inequality for \eqref{eq.prova}; for ease of notation we shall denote, from now on,
\begin{equation}\label{defHHH}
\mathcal H(s):= s + \widetilde{{\mathcal L}_h} \, H_{b,\eps}(s)\,,\qquad s\in\R\,.
\end{equation}

\begin{lemma}\label{caccioppoli lemma}
There exists a constant $c$, depending only on $n,p$ and $\Lambda$, such that, if $v$ is a solution to \eqref{eq.prova} in a cylinder $Q=B \times \Gamma$, then
\begin{align}\label{caccioppoli estimate}
& \sup_{\tau\in \Gamma}  \frac{\widetilde{{\mathcal L}_h}}{|\Gamma|} \mean{B} \Bigl[\int_{k}^{v} H_{b,\eps}'(\xi)  (\xi-k)_+ \, d\xi  \, \phi^p\Bigr](\cdot,\tau) \dx\notag
\\ \notag & \qquad  + \sup_{\tau\in \Gamma}  \frac{1}{|\Gamma|} \mean{B}\big[(v-k)_+^2 \phi^p\bigr](\cdot,\tau) \dx + \mean{Q} \big| D[ (v-k)_+ \phi ]\big|^p \dx\dt
\\ & \qquad \qquad \leq c \, \mean{Q}\Big[ (v-k)_+^p |D\phi|^p + (v-k)_+^2 \left( \partial_t \phi^p\right)_+  \Big]\dx \dt \notag
\\ & \qquad  \qquad \qquad \quad + c \, \widetilde{{\mathcal L}_h}\,   \mean{Q}    \int_{k}^v H_{b,\eps}'(\xi)  (\xi-k)_+ \, d\xi  \left( \partial_t \phi^p\right)_+ \dx \dt
\end{align}
for any $k\in\R$ and any test function $\phi\in C^\infty(Q)$, such that $(v-k)_+ \phi^p$ vanishes on the parabolic boundary of $Q$.
\end{lemma}

\begin{proof}
In order to get \eqref{caccioppoli estimate}, we test, in the weak formulation of \eqref{eq.prova}, with $(v-k)_+ \phi^p\chi_{\Gamma\cap(-\infty,\tau)}$ for $\tau\in\Gamma$. The calculations are standard; we only show here how to formally treat the parabolic term (see also the proof of Lemma \ref{the f-lemma}): being $\hat Q:=Q\cap [B\times(-\infty,\tau)]$,
\begin{multline*}
\int_{\hat Q} \partial_t v \mathcal H' (v) (v-k)_+ \phi^p \dx \dt  =  \int_{\hat Q} \partial_t \left[ \int_k^v \mathcal H^\prime (\xi) (\xi-k)_+ \, d\xi \right]\phi^p \dx \dt \\
= \int_{B} \int_k^{v(\cdot,\tau)} \mathcal H^\prime (\xi) (\xi-k)_+ \, d\xi \,  \phi^p \dx- \int_{\hat Q} \int_k^v \mathcal H^\prime (\xi) (\xi-k)_+ \, d\xi \,  \partial_t \phi^p \dx \dt \,.
\end{multline*}
\end{proof}

The next lemma allows to forward information in time. The result in the case of evolutionary $p$-Laplace type equations is a standard ``Logarithmic Lemma", see for example the proof in~\cite[Chapter II]{DiBe93}.

\begin{lemma}\label{lemma: log lemma}
Let $\overline T\in(0,T_r^\om)$, for $T_r^\om$ as in \eqref{choices T_r and T_r^1}. Suppose that $v \in C(\overline{
Q_r^\om})$ solves~\eqref{eq.prova} in $Q_r^\om$ and
\[
v(x,\overline T) \leq \osc v - \frac{\omega(r)}{4}, \qquad \forall \, x \in B_{r/8}\,;
\]
let moreover $\nu^\ast\in(0,1)$. Then there exists a constant $\varsigma \in (0,1/2)$, depending only on  $n,p,\Lambda,M$ and $\nu^*$, such that, if $\hat Q := B_{r/16} \times (\overline T,T_r^\om)$, then
\begin{equation}\label{log.general}
\frac{\Big|\hat Q \cap \left\{ v \geq \osc v - \varsigma [\omega (r)]^{1+1/\alpha} \Big\} \right| }{|\hat Q| } \leq  \nu^\ast .
\end{equation}
\end{lemma}
\begin{proof}
Denote, in short, $\widetilde{\mathcal{A}}(Dv):=\widetilde{\mathcal{A}}(x,t,v,Dv)$ and recall the definition of $\mathcal H$ in \eqref{defHHH}. Consider a time independent cut-off function $\phi \in C_0^\infty(B_r)$, $0\leq \phi \leq 1$, such that
$$\phi \equiv 1 \quad \textrm{in} \ B_{r/16}  \quad \textrm{and} \quad \phi = 0 \quad \textrm{on} \ \partial B_{r/8} \quad \textrm{with} \quad  \left| 
D \phi \right| \leq 32/r\,.$$
Take
$$0<S^+: = \frac{\left[ \omega (r) \right]^{1+1/\alpha} }{8} \leq \frac{\omega(r)}{4} \qquad \textrm{and} \qquad k = \osc v - S^+,$$
and define the logarithmic function
$$\varPsi(v) = \biggl[\ln \left( \frac{S^+}{S^+ - (v-k)_+ + \varsigma S^+} \right)\biggr]_+, \quad \varsigma \in (0,1/2) \ \ \textrm{to be fixed.}$$
We only have $\varPsi(v) \neq 0$ when
$$S^+ > S^+ - (v-k)_+ + \varsigma S^+\quad \Longleftrightarrow\quad v > \osc v -  \frac{1 - \varsigma}{8} \left[ \omega (r) \right]^{1+1/\alpha} =: v_-   \,.$$
Note, in particular, that $v_- > \osc v - \omega(r)/4$ and that $v_--k=\varsigma S^+$. We have, formally,
$$\varPsi^\prime (v) = \chi_{\{ v > v_-\}} \frac{1}{S^+ - (v-k)_+ + \varsigma S^+}$$
and
\begin{eqnarray*}
\varPsi^{\prime \prime} (v) & = & \delta_{v-v_-} \frac{1}{S^+ - (v-k)_+ + \varsigma S^+} + \chi_{\{ v > v_-\}} \frac{1}{(S^+ - (v-k)_+ + \varsigma S^+)^2}
\\
& = &  \delta_{v-v_-} \frac{1}{S^+} + \chi_{\{ v > v_-\}} \frac{1}{(S^+ - (v-k)_+ + \varsigma S^+)^2}=  \frac{\delta_{v-v_-}}{S^+} +\big[\varPsi^{\prime } (v)\big]^2,
\end{eqnarray*}
where $\delta_{v-v_-}$ is the Dirac delta centered in $v-v_-$. Testing formally the equation with $\eta = \varPsi^\prime (v) \varPsi(v) \phi^p\chi_{(\overline T, \tau)}(t)$, for $\tau \in (\overline{T},T_r^\om]$, we have
\[
- \int_{B_{r/8}\times(\overline T, \tau)}  \langle\widetilde{\mathcal{A}}(Dv), D \eta \rangle \dx \dt = \int_{B_{r/8}\times(\overline T, \tau)}  \partial_t \mathcal H(v) \eta \dx\dt\,.
\]
The choice of the test function is admissible  after  a suitable mollification in time, following the same steps as in the end of the proof of Lemma~\ref{the f-lemma}, when treating the first integral. For the time term, we have
\[
\partial_t \mathcal{H}(v) \varPsi^\prime (v) \varPsi(v) = \partial_t \int_{v_-}^v \mathcal{H}^\prime(\xi) \varPsi^\prime (\xi) \varPsi(\xi) \, d\xi
 \]
and integration by parts gives  that
\[
\int_{B_{r/8}\times(\overline T, \tau)} \partial_t \mathcal{H}(v) \varPsi^\prime (v) \varPsi(v) \phi^p \dx \dt = \int_{B_{r/8}} \int_{v_-}^{v(\cdot,\tau)}
\mathcal{H}^\prime(\xi) \varPsi^\prime (\xi) \varPsi(\xi) \, d\xi \ \phi^p \dx \bigg|_{t=\overline T}^{\tau} \,,
\]
since $\phi$ is time independent; here, we have also used the fact that $v \in C(\overline {Q_r^\om})$. Since $v \leq v_-$ on $B_{r/8} \times \{\overline T\}$, we have that
\[
 \int_{B_{r/8}} \int_{v_-}^{v(\cdot,\overline T)} \mathcal{H}^\prime(\xi) \varPsi^\prime (\xi) \varPsi(\xi) \, d\xi\  \phi^p \dx = 0\,.
\]
Therefore
\[
\int_{B_{r/8}\times(\overline T,\tau)} \partial_t \mathcal{H}(v) \varPsi^\prime (v) \varPsi(v) \phi^p \dx \dt  =  \int_{B_{r/8}} \int_{v_-}^{v(\cdot,\tau)}
\mathcal{H}^\prime(\xi) \varPsi^\prime (\xi) \varPsi(\xi) \, d\xi\ \phi^p \dx
\]
and since $\mathcal H^\prime \geq 1$ and  $\varPsi(v_-)=0$, we obtain that
\[
\int_{B_{r/8}} \varPsi^2 (v(x,\tau))  \phi^p \dx \leq 2 \int_{B_{r/8}\times(\overline T,\tau)} \partial_t \mathcal{H}(v) \varPsi^\prime (v) \varPsi(v) \phi^p \,
dx \dt.
\]
As for the elliptic term, we get, from \eqref{tilde inf}, since $\varPsi(v)\delta_{v-v_-}=0$,
\begin{align*}
-\int_{B_{r/8}\times(\overline T,\tau)}\langle\widetilde{\mathcal{A}}(Dv), D\eta \rangle &\dx \dt  =- \int_{B_{r/8}\times(\overline T,\tau)}\langle\widetilde{\mathcal{A}}(Dv),D\phi^p\rangle  \, \varPsi^\prime(v) \varPsi(v)  \dx \dt\\
&\quad- \int_{B_{r/8}\times(\overline T,\tau)}\langle\widetilde{\mathcal{A}}(Dv),Dv\rangle \big(1+\varPsi(v)\big) \left[ \varPsi^\prime(v) \right]^2 \phi^p  \dx \dt \\
&\leq c(p,\Lambda)  \int_{B_{r/8}\times(0,T_r^\om)}  \varPsi(v) \left[ \varPsi^\prime(v) \right]^{2-p} |D\phi|^p \dx \dt\\
&\quad -c(p,\Lambda)\int_{B_r\times(\overline T,\tau)} |Dv|^p (1+\varPsi(v)) \left[ \varPsi^\prime(v) \right]^2 \phi^p  \dx \dt\,,
\end{align*}
using Young's inequality. We thus obtain, discarding the negative term on the right-hand side,
\[
\int_{B_{r/8}} \varPsi^2 (v(\cdot,\tau))  \phi^p \dx \leq c  \int_{Q_r^\om}   \varPsi(v) \left[ \varPsi^\prime(v) \right]^{2-p}|D\phi|^p \dx \dt\,;
\]
this holds for all $\tau \in (\overline T, T_r^\om]$.  The very definitions of $\varPsi$ and $T_r^\om$ then imply
\begin{eqnarray*}
\int_{B_{r/16}} \left[ \varPsi(v(\cdot,\tau)) \right]^2 \dx & \leq &  c  \, \frac{| B_{r/8} |\,T_r^\om}{r^p} \ln \frac{1}{\varsigma} \ (2S^+)^{p-2}
\\ & \leq &  c \, M | B_{r/16} | \ln \frac{1}{\varsigma}\,,
\end{eqnarray*}
since $(v-k)_+\leq S^+$ and
\[
r^{-p} \, T_r^\om \, (2S^+)^{p-2} = 2^{p-2}  M [ \omega(r) ]^{(2-p)(1+1/\alpha)} \, \left( \frac{\left[ \omega (r) 
\right]^{1+1/\alpha} }{8} \right)^{p-2} = 4^{2-p}M\,.
\]
Moreover, the left-hand side can be bounded below as
$$\int_{B_{r/16}} \left[ \varPsi(v(\cdot,\tau)) \right]^2 \dx \geq \left| B_{r/16} \cap \left\{ v(\cdot , \tau ) \geq \osc v - \varsigma S^+ \right\} 
\right| \left( \ln \frac{1}{2\varsigma} \right)^2$$
and we conclude, recalling the definition of $S^+$, that
\[
\frac{\left| B_{r/16} \cap \left\{ v(\cdot , \tau ) \geq \osc v - \varsigma [ \omega (r) ]^{1+1/\alpha} \right\} \right| }{| B_{r/16}  | } \leq cM\, \frac{\ln \frac{1}{\varsigma}}{\, \ln \frac{1}{2\varsigma}\,} = \nu^\ast\,,
\]
for a convenient choice of $\varsigma$. Finally, integrate in time to obtain \eqref{log.general} and complete the proof.
\end{proof}

\subsection{Supersolutions of evolutionary $p$-Laplace equations}

We recall that a weak supersolution to 
\begin{equation}\label{ssol}
 \partial_t v - \mathrm{div} \, \hat{\mathcal{A}}(x,t,v,Dv)= 0\qquad \text{in $B\times \Gamma$}\,,
\end{equation}
$B$ open set and $\Gamma$ open interval, where $\hat {\mathcal A}$ has the same structure of $\widetilde{\mathcal A}$ (and $\mathcal A$), is a function $w\in V^{2,p}(B\times \Gamma)$ satisfying
\[
\int_{\mathcal K}[ w\,\varphi](\cdot,\tau) \dx\biggr|_{\tau=t_1}^{t_2} +\int_{\mathcal K\times[t_1,t_2]} \big[-w\, \partial_t\varphi  +\langle \mathcal{A}(\cdot, \cdot, w,Dw) ,D\varphi\rangle\big]\dx\dt\geq0
\]
for all $\mathcal K\Subset B$, almost every $t_1,t_2\in\R$ such that $[t_1,t_2]\Subset\Gamma$ and for every test function $\varphi \in L^p_{\loc}(\Gamma;W^{1,p}_0(\mathcal K))$ such that $\partial_t\varphi\in L^2(\mathcal K\times[t_1,t_2])$ and $\varphi\geq0$. Analogously, $w$ is a weak subsolution if the quantity on the left-hand side in \eqref{ssol} is non-positive for any such test function. The following simple lemma is one of the keys in our proof of the interior continuity. 

\begin{lemma} \label{the f-lemma}
If $k<b-\eps $ and $v$ is a weak solution of \eqref{eq.prova} in $Q_r^\om$, then $(k-v)_+$ is a weak subsolution and $\min (k,v) = k - (k-v)_+$ is a weak 
supersolution of \eqref{ssol} in $Q_r^\om$, where $\hat {\mathcal A}$ has the same structure of $\mathcal A$.

\end{lemma}

\begin{proof}
Let $\mathcal K\Subset B_r$, $[t_1,t_2]\Subset (0,T^\om_r)$, call $\mathcal Q:=\mathcal K\times[t_1,t_2]$ and let $\varphi$ be a test function as above, in particular non-negative; in order to simplify the proof we suppose $\varphi\equiv 0$ in $\mathcal K\times\{t_1,t_2\}$, it will be easy to deduce the proof also in the general case. Set
\[
\phi_{k,\epsilon}(\xi) = \min \biggl\{ \frac{(k-\xi)_+}{\epsilon}, 1\biggr\}, \qquad\text{for $\epsilon \in (0,1)$\,,}
\]
and test equation \eqref{eq.prova} with $\phi_{k,\epsilon}(v) \, \varphi$. Formally, the time derivative terms give
\begin{eqnarray}\label{first.parabolic.sub}
\int_{\mathcal Q} \partial_t v \phi_{k,\epsilon}(v)\, \varphi \dx dt & = & - \int_{\mathcal Q} \partial_t \int_v^k  \phi_{k,\epsilon}(\xi)\, d\xi \, \varphi \dx \dt
\notag\\
& = & \int_{\mathcal Q} \int_v^k  \phi_{k,\epsilon}(\xi)\, d\xi \, \partial_t  \varphi \dx \dt\notag\\
& \stackrel{\epsilon \downarrow 0}{\longrightarrow} & \int_{\mathcal Q} (k-v)_+ \, \partial_t  \varphi \dx \dt\,,
\end{eqnarray}
by the dominated convergence theorem, and
\begin{eqnarray*}
-\int_{\mathcal Q} \partial_t v H_{b,\eps}^\prime (v) \phi_{k,\epsilon}(v)\, \varphi \dx dt & = &  \int_{\mathcal Q} \partial_t \int_v^k H_{b,\eps}^\prime (\xi)
\phi_{k,\epsilon}(\xi)\, d\xi \, \varphi \dx \dt \\
& =& - \int_{\mathcal Q}  \int_v^k H_{b,\eps}' (\xi) \phi_{k,\epsilon}(\xi)\, d\xi \, \partial_t \varphi \dx \dt \\
& = & 0\,,
\end{eqnarray*}
since ${\rm supp}\, H'_{b,\eps}\subset(b-\eps,b+\eps)$ does not intersect the integration interval $(v,k)$ due to the fact that we 
assume $k<b-\eps$. As for the elliptic part, noting that $\phi^\prime_{k,\epsilon}(v)=-\frac{1}{\epsilon} \chi_{\{ k-\epsilon < v < k\}} \leq 0$ and hence
\[
\int_{\mathcal Q}\big\langle \widetilde{\mathcal{A}}(x,t,v,Dv), D\phi_{k,\epsilon}(v) \big\rangle\varphi\dx \dt  \leq 0\,,
\]
we obtain
\begin{align*}
\int_{\mathcal Q}\big\langle \widetilde{\mathcal{A}}(&x,t,v,Dv), D\big[\phi_{k,\epsilon}(v)\, \varphi\big] \big\rangle\dx \dt  \\
&\leq  \int_{\mathcal Q} \ \ \,  \big\langle\widetilde{\mathcal{A}}(x,t,v,Dv),  D \varphi\big\rangle \phi_{k,\epsilon}(v) \dx \dt\\
&\stackrel{\epsilon \downarrow 0}{\longrightarrow}  \int_{\mathcal Q}  \big\langle\widetilde{\mathcal{A}}(x,t,v,Dv),  D \varphi \big\rangle\chi_{\{v<
k\}}\dx \dt\,,
\end{align*}
yielding the conclusion for $(k-v)_+$, once we define $\hat{\mathcal{A}}(x,t,w,\xi):=-\widetilde{\mathcal{A}}(x,t,k-w,-\xi)$. The second 
result follows immediately from this one.

\vspace{3mm}

To justify the above calculations, we demonstrate how to rigorously test equation~\eqref{eq.prova} with a test function depending on $v$ itself; indeed, there is a well recognized difficulty concerning the time regularity of solutions and one has to suitably mollify the test function in time. To this end, take $\rho_h(s)$, for $h\in(0,1)$, the standard symmetric positive mollifier, with support in $(-h,h)$ and denote, for any function $\theta:\R\to\R$, its mollification by $\theta_h:=\theta\ast \rho_h$. If $\theta$ is not defined over $\R$, extend it to zero elsewhere before mollifying. Let $f:\R^+\to\R^+$ be any Lipschitz function; note that, for $\mathcal H(\cdot)$ defined in \eqref{defHHH}, we have that $v\mapsto \mathcal H(v)$ is an increasing function. Therefore, consider as a test function in \eqref{eq.prova} the function
\[
\phi\equiv \phi_h:=\Big[f\big(\mathcal H^{-1}([\mathcal H(v)]_h) \big)\varphi\Big]_h\,,
\]
for $h>0$ small, where $[\mathcal H(v)]_h$ is the convolution of $\mathcal H(v)$ with respect to the time variable and $\varphi$ is as in the beginning of the proof. Note finally that since $\rho_h$ is symmetric, then $\int fg_h\dt=\int f_hg\dt$ by Fubini's theorem; therefore, first using this fact and subsequently integrating by parts, we get
\begin{eqnarray*}
-\int_{B}\int_\Gamma \mathcal H(v) \partial_t\phi_h\dt\dx&=& \int_{B}\int_\Gamma\partial_t [\mathcal H(v)]_hf\big(\mathcal H^{-1}([\mathcal H(v)]_h) \big)\varphi\dt\dx\\
&=&- \int_{B}\int_\Gamma\partial_t \int_{[\mathcal H(v)]_h}^{\mathcal H(k)}f\big(\mathcal H^{-1}(\zeta) \big)\,d\zeta\,\varphi\dt\dx\\
&= &\int_{B}\int_\Gamma\int_{[\mathcal H(v)]_h}^{\mathcal H(k)}f\big(\mathcal H^{-1}(\zeta) \big)\,d\zeta\,\partial_t\varphi\dt\dx\\
&\stackrel{h \downarrow 0}{\longrightarrow}&  \int_{\mathcal Q}\int_{\mathcal H(v)}^{\mathcal H(k)}f\big(\mathcal H^{-1}(\zeta) \big)\,d\zeta\,\partial_t\varphi\dx\dt\\
&=&  \int_{\mathcal Q}\int_{v}^{k}f(\xi)\big(1+H_{b,\eps}'(\xi)\big)\,d\xi\,\partial_t\varphi\dx\dt\,,
\end{eqnarray*}
recalling the definition of $\mathcal H$. In the case $f(\xi)= \phi_{k,\epsilon}(\xi)$, for $\epsilon \in (0,1)$, we then also take the limit for $\epsilon \downarrow0$ as in \eqref{first.parabolic.sub} and we discard the remaining null term. As for the elliptic part we may use dominated convergence, together with the fact that $v \in L^p(t_1,t_2;W^{1,p}(\mathcal K))$, and send first $h$ and then $\epsilon$ to zero to follow the formal calculation in the beginning of the proof.
 \end{proof}
 
\subsection{Harnack estimates} 
The following weak Harnack inequality for supersolutions is Theorem 1.1 of \cite{Kuu08}.
\begin{theorem}[Weak Harnack inequality]\label{Har.Kuusi}
Let $v$ be a non-negative continuous weak supersolution to
\begin{equation}\label{supersolution}
\partial_t v-\divergence \mathcal{A}(x,t,v,Dv)=0\qquad\text{in}\quad B_{4R_0}(x_0)\times(0,T)\,,
\end{equation}
with $\mathcal A$ satisfying \eqref{p.laplacian.assumptions}. Then there exist constants $c_1$ and $c_2$, both depending only on $n,p$ and $\Lambda$, such
that for every $0<t_1<T$ we have
\begin{equation}\label{w.Harnack}
\mean{B_{R_0}(x_0)}v(x,t_1)\dx \leq \frac12\biggl(\frac{c_1 R_0^p}{T-t_1}\biggr)^{1/(p-2)}+c_2\inf_{\mathcal Q}v\,,
\end{equation}
where $\mathcal Q:=B_{2R_0}(x_0)\times (t_1+\tau/2,t_1+\tau)$ and
\begin{equation}\label{time.Harnack}
\tau:=\min\biggl\{T-t_1,c_1R_0^p\biggl(\ \mean{B_{R_0}(x_0)}v(x,t_1)\dx\biggr)^{2-p}\biggr\}\,.
\end{equation}
\end{theorem}
The factor $1/2$ in the above theorem is not present in the formulation of \cite{Kuu08}. Nonetheless, this constant is insignificant as it only increases the value of the constants $c_1$ and $c_2$, a fact that can be easily deduced from the proof in~\cite{Kuu08}. For related results, see the recent interesting monograph by DiBenedetto, Gianazza and Vespri~\cite{DiBeGianVesp12}, and also~\cite{DiBeGianVesp08}, by the same authors, about the Harnack inequality for {\em weak solutions}.

\vspace{3mm}

The next proposition, which encodes the decay rate of supersolutions, follows from the iteration of the previous theorem; see \cite[Corollary 3.4]{GSV} for a very similar statement.
\begin{proposition}[Decay of positivity]\label{lemma:decay of pos}
Let $v$ be a non-negative continuous weak supersolution to \eqref{supersolution} in $B_{4R_0}(x_0) \times (t_0,t_0 + T)$. Then there exists a constant $c_3$, depending only on $n,p$ and $\Lambda$, such that, if
\begin{equation}\label{starting.level}
\inf_{x\in B_{2R_0}(x_0)}v(x,t_0)\geq k  
\end{equation}
for some level $k>0$, then
\[
\inf_{x\in B_{2R_0}(x_0)} v(x,t) \geq \lambda(t) := \frac{k}{c_3} \left(1+c_3(p-2)k^{p-2}\frac{t-t_0}{R_0^p}\right)^{-\frac{1}{p-2}}
\]
for all $t \in (t_0,t_0+T]$.
\end{proposition}

\begin{proof}
Suppose, without loss of generality, that $t_0=0$. Define inductively
\[
\tau_0:=t_0=0,\qquad\tau_j := c_1R_0^p \sum_{\ell=1}^j \biggl(\ \mean{B_{R_0}(x_0)}\min\big\{v(\cdot,\tau_{\ell-1}),(2c_2)^{-\ell}k\big\}\dx\biggr)^{2-p},
\]
for all indices $j$ such that $\tau_j\leq T$, say $j\in\{1,\dots,\bar\jmath\}$, and where $c_2$ is the constant of Theorem \ref{Har.Kuusi}. Note that, for $i\in\{1,\dots,\bar\jmath\}$, there holds
\[
\biggl(\frac{c_1R_0^p}{\tau_i-\tau_{i-1}}\biggr)^{\frac{1}{p-2}}= \mean{B_{R_0}(x_0)}\min\big\{v(\cdot,\tau_{i-1}),(2c_2)^{-i}k\big\}\dx\,;
\]
hence, since $\tau$ in \eqref{time.Harnack} turns out to be, in our case,  exactly $\tau_i-\tau_{i-1}$,  Harnack estimate~\eqref{w.Harnack} applied to the supersolution $v_i:=\min\{v,(2c_2)^{-i}k\}$ gives
\begin{equation}\label{bound.below.Harnack}
\inf_{B_{2R_0}(x_0) \times ((\tau_{i-1}+\tau_i)/2,\tau_i)} v_i \geq\frac{1}{2c_2}\mean{B_{R_0}(x_0)}v_i(\cdot,\tau_{i-1})\dx\geq\frac{k}{(2c_2)^i}\,,
\end{equation}
and the last inequality holds if $\inf_{B_{R_0}(x_0) }v_i(\cdot,\tau_{i-1})\geq (2c_2)^{-(i-1)}k$. Using an iterative argument, starting from \eqref{starting.level}, we see that \eqref{bound.below.Harnack} holds for any $j \in \{1,\dots,\bar\jmath\}$. This means that, for such a $j$, we have $v_j(x,\tau_j)=(2c_2)^{-j}k$ in $B_{R_0}(x_0)$ and $\tau_j = c_1k^{2-p}R_0^p\sum_{\ell=1}^j (2c_2)^{\ell(p-2)}$. Therefore,
\[
\int_0^j (2c_2)^{s(p-2)}\ds \leq \frac{\tau_j}{c_1k^{2-p}R_0^p} \leq   \int_{1}^{j+1} (2c_2)^{s(p-2)}\ds
\]
and we thus obtain a lower and an upper bound for $\tau_j$:
\[
\frac{(2c_2)^{j(p-2)}-1}{(p-2)\ln(2c_2)} \leq \frac{\tau_j}{c_1k^{2-p}R_0^p} \leq 2c_2 \frac{(2c_2)^{j(p-2)}-1}{(p-2)\ln(2c_2)}\,.
\]
The bound from below gives
\[
(2c_2)^{-j} \geq \left(1 + (p-2)  \frac{\ln(2c_2)}{c_1} \frac{\tau_j }{k^{2-p}R_0^p}\right)^{-1/(p-2)} \geq\frac{c_3}{k} \lambda(\tau_j),
\]
provided that $c_3\geq\ln(2 c_2)/c_1$.   Finally, taking into account that $v_i\leq v$, another application of~\eqref{bound.below.Harnack}, for an appropriate $R_0$, and with starting time $\tau_{j-1}$, together with a simple covering argument, shows that
\[
\inf_{B_{2R_0}(x_0) \times (\tau_{j-1},\tau_j)} v \geq \frac1{2c_2} c_3 \lambda(\tau_{j-1}) \geq \lambda(\tau)
\]
whenever $\tau \in (\tau_{j-1},\tau_j)$, provided that $c_3 \geq 2c_2$. Clearly, at this point, taking $c_3 := 2c_2\geq\ln(2 c_2)/c_1$ finishes the proof.
\end{proof}

\section{Reducing the oscillation}\label{reducing}
Recalling now the definitions of $\widetilde Q_r^\om$ and $Q_r^\om$ from subsection \ref{spgeo}, we suppose that $w$ is a weak solution to \eqref{eq:mollified eq} in $Q_r^\om$.

\subsection{Basic reductions}\label{reductions}
Define
\begin{equation}\label{v}
 v(x,t):= w(x,t) - \inf_{Q_r^\om} w \qquad \mathrm{and} \qquad b:=a-\inf_{Q_r^\om} w\,.
\end{equation}
Then $\sup v = \osc v =\osc w$, $\inf v = 0$, these quantities being meant over $Q_r^\om$, and
\begin{equation} \label{the equation for v}
\partial_t v - \mathrm{div} \, \widetilde{\mathcal{A}}(x,t,v+\inf_{Q_r^\om} w,Dv)  = -\widetilde{{\mathcal L}_h}\, \partial_t H_{b,\eps}(v)\,,
\end{equation}
$\widetilde{{\mathcal L}_h} \in [0,1]$. From now on we shall also suppose that
\begin{equation}\label{ass.trivial}
\osc v:=\osc_{Q_r^\om}v\geq\omega(r) \qquad \text{and} \qquad \eps<\frac{\omega(r)}{8}\,.
\end{equation}
Note that if $b \notin [0, \osc v ]$, we then  have
$$\partial_t v - \mathrm{div} \, \widetilde{\mathcal{A}}(x,t,v+\inf_{Q_r^\om} w,Dv) =0 \qquad \mathrm{in} \; \, Q_r^\om$$
for $\eps$ small enough, and the oscillation reduction follows by the well-known argument of DiBenedetto, see~\cite{DiBe93, Urba08}. In this case, even if the modulus of continuity is H\"older, we will not make use of this information since the intrinsic geometry we are using does not allow us to reproduce the estimates of \cite{DiBe93, Urba08}. We, instead, observe that our reasoning also works in the case of evolutionary $p$-Laplace type equations since the phase transition term $\widetilde{{\mathcal L}_h}\,\partial_t H_{b,\varepsilon}$ only appears as an inhomogeneous term in our calculations, and in particular it works for $\widetilde{{\mathcal L}_h} = 0$. 

Thus we may assume from now on $b \in [0, \osc v ]$. 
If $b \in \left[ 0, \frac{\osc v}{2} \right]$, we can consider $\bar{v}=\osc v - v$ and $\bar{b}=\osc v - b$ instead, and 
then
\[
\partial_t \bar{v} - \mathrm{div} \,  \bar{\mathcal{A}}(x,t,\bar v,D \bar v) = - \widetilde{{\mathcal L}_h}\, \partial_t H_{\bar{b},\eps}(\bar{v}) 
\]
with $\bar{b} \in \left[ \frac{\osc \bar{v}}{2}, \osc \bar{v} \right]$. Here
$$
\bar{\mathcal{A}}(x,t,\bar v,D \bar v) = - \widetilde{\mathcal{A}}(x,t,-\bar v +\sup_{Q_r^\om} w,-D \bar v)\,,
$$
which has the same structure as $\mathcal{A}$. Consequently we can further assume that
\begin{equation*}
b \in \left[ \frac{\osc v}{2}, \osc v \right]. 
\end{equation*}
Let us, finally, introduce the Sobolev conjugate exponent of $p$, $\kappa p$, where 
\begin{equation} \label{eq:kappa}
\kappa := \begin{cases}
  \frac{n}{n-p} &\qquad  \text{for }p<n\,,   \\[3pt]
  \mbox{any number $>1$} &\qquad \text{for }p=n\,, \\[3pt]
   +\infty &\qquad\text{for } p>n\,;
\end{cases}
\end{equation}
$\alpha$, appearing in \eqref{ratio.kappa}, will be related to $\kappa$ in the following way: 
$$\frac1\alpha=1+\frac{\kappa}{\kappa-1}\,.$$ 
From now on, it will be more convenient for our purposes to work with $\kappa$.

\vspace{3mm}

Now we fix the classical alternative. Clearly one of the following two options must hold: for $\varepsilon_1$  a free parameter, to be fixed in due course, either
\begin{equation}\label{alt.1}
\left| \widetilde Q_r^{\omega(\cdot)}\cap \left\{  v \geq \frac{\osc v}{4} \right\} \right| > \varepsilon_1 \left[ \omega (r) \right]^{1+\frac{\kappa}{\kappa-1}} \big| \widetilde Q_r^{\omega(\cdot)} \big|
\tag{\text{Alt.\,1}}
\end{equation}
or
\begin{equation}\label{alt.2}
\left|  \widetilde Q_r^{\omega(\cdot)}\cap\left\{  v \geq \frac{\osc v}{4} \right\} \right| \leq \varepsilon_1 \left[ \omega (r) \right]^{1+\frac{\kappa}{\kappa-1}} \big| \widetilde Q_r^{\omega(\cdot)} \big|
\tag{\text{Alt.\,2}}
\end{equation}
holds true. We analyze separately the two different cases.

\subsection{The first alternative}
Consider first the case where \eqref{alt.1} holds. Then there exists $ t_r^1 \in (0,\widetilde T_r^\om)$ such that
\begin{equation}\label{alt.1.time}
\left| B_{r/4}\cap\left\{  v(\cdot, t_r^1) \geq \frac{\osc v}{4} \right\}   \right| > \varepsilon_1 \left[ \omega (r) \right]^{1+\frac{\kappa}{\kappa-1}} \left| B_{r/4} \right|; 
\end{equation}
otherwise, just integrate to get a contradiction.

\vspace{3mm}

Observing that, due to \eqref{ass.trivial},
\[
\frac{\osc v}{4}<\frac{\osc v}{2}-\frac{\osc v}{8}\leq b-\frac{\omega(r)}{8}<b-\eps,
\]
we can use the weak Harnack estimate on the supersolution $\hat v:=\min\{v,\osc v/4\}$. Thus, Lemma \ref{the f-lemma}, and hence Theorem \ref{Har.Kuusi}, apply to $\hat v$:
\begin{equation}\label{*}
\mean{B_{r/4}} \hat v(x, t_r^1)\dx \leq\frac12 \biggl( \frac{c_1 (r/4)^p}{T_r^\om - t_r^1} \biggr)^{\frac{1}{p-2}} + c_2 \inf_{B_{r/2} \times ( t_r^1 +
{\tau/2}, t_r^1 + \tau ) } \hat v\,,
\end{equation}
where
$$\tau = \min \left\{ T_r^\om - t_r^1 , c_1 \Big(\frac r4\Big)^p \biggl(\mean{B_{r/4}} \hat v(x, t_r^1)\dx \biggr)^{2-p} \right\}\,.$$
Due to \eqref{alt.1.time},
\begin{equation}\label{es.es.es}
 \mean{B_{r/4}} \hat v(x, t_r^1)\dx \geq \varepsilon_1 \left[ \omega (r) \right]^{1+\frac{\kappa}{\kappa-1}} \frac{\osc v}{4} \geq  \frac{\eps_1}{4} \left[ \omega
 (r) \right]^{2+\frac{\kappa}{\kappa-1}},
\end{equation}
where the last inequality follows from \eqref{ass.trivial}.
Now, if
\begin{equation}\label{**}
T_r^\om-t_r^1 \geq  c_1 \Big(\frac r4\Big)^p  \biggl( \mean{B_{r/4}} \hat v(x, t_r^1)\dx \biggr)^{2-p},
\end{equation}
then
$$\tau = c_1 \Big(\frac r4\Big)^p  \biggl( \mean{B_{r/4}}\hat  v(x, t_r^1)\dx \biggr)^{2-p}$$
and
$$\left( \frac{c_1 \Big(\displaystyle{\frac r4}\Big)^p}{T_r^\om - t_r^1} \right)^{\frac{1}{p-2}}  \leq \left( \frac{c_1 \Big(\displaystyle{\frac r4}\Big)^p}{c_1 \displaystyle \Big(\frac r4\Big)^p  \left( \displaystyle \mean{B_{r/4}}\hat  v(x, 
t_r^1)\dx \right)^{2-p}} \right)^{\frac{1}{p-2}} =   \mean{B_{r/4}} \hat v(x, t_r^1)\dx\,.$$
So \eqref{*} reads
$$\mean{B_{r/4}} \hat v(x, t_r^1)\dx \leq 2 c_2  \inf_{B_{r/2} \times (t_r^1 + {\tau/2}, t_r^1 + \tau) } \hat v$$
and consequently, combining the previous display with \eqref{es.es.es}, we get
\begin{equation}\label{part.case1}
\frac{\eps_1}{8c_2}\left[ \omega (r) \right]^{2+\frac{\kappa}{\kappa-1}} \leq  \inf_{B_{r/2}\times ( t_r^1 + \tau/2, t_r^1 + \tau ) } \hat v\,.
\end{equation}
Hence if \eqref{**} holds, then we infer \eqref{part.case1}. Note now that, in particular,  if we fix 
\begin{equation} \label{eq:M fix}
M := 1+  \frac{\eps_1^{2-p} c_1 }{16}\geq2
\end{equation}
in the definition of $T_r^\om$, provided that $\eps_1^{p-2}\leq c_1/16$, then
\begin{equation*}
T_r^\om-\widetilde T_r^\om\geq T_r^\om-[ \omega(r) ]^{(2-p)(2+\frac{\kappa}{\kappa-1})} r^p = \eps_1^{2-p} c_1 r^p \frac{\left[ \omega (r) \right]^{(2-p)(2+\frac{\kappa}{\kappa-1})}}{16}\,.
\end{equation*}
Thus we have, by \eqref{es.es.es}, that
\begin{align*}
T_r^\om-t_r^1 \geq T_r^\om-\widetilde T_r^\om & =  c_1 \Big(\frac r4\Big)^p \left( \frac{ \eps_1}{4} \left[ \omega (r) \right]^{2+\frac{\kappa}{\kappa-1}} \right)^{2-p}
\\
&\geq c_1 \Big(\frac r4\Big)^p \left( \mean{B_{r/4}} \hat v(x, t_r^1)\dx \right)^{2-p}=\tau
\end{align*}
and hence \eqref{**} is satisfied.

\vspace{3mm}

Now the goal is to push positivity at time $t_r^1+\tau$ up to time $T_r^\om$; note that by \eqref{**} and subsequent lines, $t_r^1+\tau\leq T_r^\om$. To do this, we use Proposition~\ref{lemma:decay of pos}, with $k=\eps_1\left[ \omega (r) \right]^{2+\frac{\kappa}{\kappa-1}}/(8c_2)$, to obtain
\begin{align*}
\inf_{B_{r/2}\times (t_r^1 + {\tau/2}, T_r^\om)}\hat v &\geq \frac{k}{c_3}\left(1+c_3(p-2)k^{p-2}\frac{T_r^\om-(t_r^1 + {\tau/2} )}{(r/4)^p}\right)^{-\frac{1}{p-2}}\\
&\geq \frac{\eps_1}{8c_2c_3}\left[ \omega (r) \right]^{2+\frac{\kappa}{\kappa-1}} \big(1+\tilde c\,c_3(p-2)\big)^{-\frac{1}{p-2}},
\end{align*}
since
\[
T_r^\om-\Big(t_r^1 + \frac{\tau}{2} \Big)\leq T_r^\om\leq \frac{c_1}{\,8\eps_1^{p-2}}  \left[ \omega (r) \right]^{(2-p)(2+\frac{\kappa}{\kappa-1})}r^p=\tilde c\,k^{2-p}r^p\,,
\]
$\tilde c$ depending on $p,c_1,c_2$ and hence, ultimately, only on $n,p$ and $\Lambda$. Recalling that, clearly, $\hat v\leq v$, and noting that $\tau\leq T_r-\widetilde T_r^\om$ and $\widetilde T_r^\om\leq T_r^\om/2$, by \eqref{**} and \eqref{eq:M fix}, we conclude that the infimum of $v$ has been lifted and thus we have reduced the oscillation: we have indeed proved that
\begin{equation}\label{conclusion 1}
\text{\eqref{ass.trivial} and \eqref{alt.1}}\quad\Longrightarrow\quad\osc_{B_{r/4} \times (3T_r^\om/4, T_r^\om)} \ v \leq
\osc_{Q_r^\om} v - \theta_1 \left[ \omega (r) \right]^{2+\frac{\kappa}{\kappa-1}}\,,
\end{equation}
with $\theta_1\equiv \theta_1(n,p,\Lambda,\eps_1)\in(0,1)$.

\subsection{The second alternative} Let us now consider the case when the second alternative \eqref{alt.2} holds:
$$\left| \widetilde Q_r^{\omega(\cdot)}\cap\left\{  v \geq \frac{\osc v}{4} \right\} \right| \leq \varepsilon_1\left[ \omega (r) \right]^{1+\frac{\kappa}{\kappa-1}} \big| \widetilde Q_r^{\omega(\cdot)}\big|\,.$$
We shall use this information as a starting point for a De Giorgi-type iteration, where we fix the sequence of nested 
cylinders as
\begin{equation*}
 U_j = B_j \times \Gamma_j := B_{(1+2^{-j}) r/8} \times \left( \frac{1-2^{-j}}{2} \, \widetilde T_r^\om, \widetilde T_r^\om\right),
\end{equation*}
and we consider cut-off functions $\phi_j$ such that
$$\phi_j \equiv 1 \quad \textrm{in} \ U_{j+1} \qquad \textrm{and} \qquad \phi_j = 0 \quad \textrm{on} \ \partial_p U_j\,,$$
with
\begin{equation}\label{test.gra}
 \left( \partial_t \phi_j^p\right)_+ \leq \frac{c\,2^j}{\widetilde T_r^\om} \qquad \textrm{and} \qquad  \left| D \phi_j \right| \leq \frac{c\,2^{j}}{r}\,.
\end{equation}
Using then the energy estimate \eqref{caccioppoli estimate}, with $\kappa$ defined in~\eqref{eq:kappa} (with the formal agreement that $1/\infty=0$ and 
\[
\biggl(\ \mean{\, B_j} \left[(v-k)_+  \phi_j\right]^{\kappa p} \dx \biggr)^{1/\kappa}:=\big\|(v-k)_+  \phi_j\bigr\|_{L^\infty(B_j)}^p\qquad\text{when $\kappa=\infty$}),
\] 
we infer
\begin{align*}
\mean{\, U_{j+1}} &(v-k)_+^{2 (1-1/\kappa)+p} \dx\dt\notag\\
&\leq \mean{\, U_j} \left[ (v-k)_+^2  \phi_j^p \right]^{(1-1/\kappa) }(v-k)_+^p  \phi_j^p   \dx \dt\notag\\
&\leq  \mean{\, \Gamma_j} \biggl[\mean{\, B_j} (v-k)_+^2  \phi_j^p \dx \biggr]^{1-1/\kappa}     \biggl[\mean{\, B_j} \left[(v-k)_+  \phi_j\right]^{\kappa p} \dx \biggr]^{1/\kappa}  \dt \notag\\
&\leq c\,\big[\widetilde T_r^\om\big]^{1-1/\kappa} \biggl[ \sup_{t \in \Gamma_j} \frac{1}{\widetilde T_r^\om}  \mean{\, B_j} \big[(v-k)_+^2  \phi_j^p\big](\cdot,t) \dx \biggr]^{1-1/\kappa}\times\notag
\\
& \qquad\times r^p \mean{\, U_j}  \big| D \left[ (v-k)_+ \phi_j \right] \big|^p  \dx \dt\notag\\
&\leq c\, r^p \big[\widetilde T_r^\om\big]^{1-1/\kappa} \biggl[\ \mean{\, U_j} \Bigl(    (v-k)_+^p|D\phi_j|^p \notag\\
&\hspace{1cm}+\Big[ (v-k)_+^2 +\widetilde{{\mathcal L}_h}\, (b+\eps-k)_+ \chi_{\{ v \geq k\}} \Big] \left( \partial_t \phi_j^p\right)_+ \Bigr)\dx\dt \biggr]^{2-1/\kappa},
 \end{align*}
using H\"older's inequality and Sobolev's embedding. 
The next step is to choose the levels
$$k_j := \osc v - \frac{1+2^{-j}}{4} \omega (r)\,.$$
We have $ k_j > \frac{\osc v}{4}$, since $\omega (r) \leq \osc v$, and the relations
\begin{align*}
&(v-k_j)_+ \geq  (k_{j+1} - k_j)  \chi_{\{ v \geq k_{j+1} \} } = 2^{-j-3} \omega (r) \chi_{\{ v \geq k_{j+1} \} }\,,\\[3pt]
&(v-k_j)_+ \leq [\omega (r)]\chi_{\{ v \geq k_j \} }\,,\\[3pt]
&(b+\eps - k_j)_+ \leq \omega (r) \qquad ( \text{since $b\leq \osc v$ and $\eps\leq \omega(r)/8$}  )\,.
\end{align*}
We go back to the iteration inequality, with the notation
$$A_j := \frac{\left| U_j \cap \left\{ v \geq k_j \right\}  \right| }{\left| U_j \right|}\,,$$
to obtain, using the definition of $T_r^\om$ \eqref{choices T_r and T_r^1} and \eqref{test.gra}
\begin{align*}
\big( 2^{-j-3} & \omega (r) \big)^{2  (1-1/\kappa)+p} A_{j+1}\\
& \leq  c r^p \big[\widetilde T_r^\om\big]^{1-1/\kappa} \left[   2^j \frac{\omega (r)}{\widetilde T_r^\om} + 2^j \frac{\omega (r)^2}{\widetilde T_r^\om} + 2^{jp} \frac{\omega
(r)^p}{r^p}\right]^{2-1/\kappa} A_j^{2-1/\kappa}\\
& \leq  c^j r^p \big[r^p \left[ \omega(r) \right]^{2-p}\big]^{1-1/\kappa} \left[  \frac{\omega (r)^{p-1}}{r^p} +  \frac{\omega (r)^p}{r^p}\right]^{2-1/\kappa}
A_j^{2-1/\kappa} \\
&\leq  c^j  \left[ \omega(r) \right]^{(2-p)(1-1/\kappa) +(p-1)(2-1/\kappa)} A_j^{2-1/\kappa}\,.
\end{align*}
Note here that we also appealed to the fact that $0\leq \widetilde{{\mathcal L}_h} \leq 1$. Thus,
\begin{align*}
A_{j+1}  &\leq  c_0^j  \left[ \omega(r) \right]^{(2-p)(1-1/\kappa) + (p-1)(2-1/\kappa) - p - 2  (1-1/\kappa)} A_j^{2-1/\kappa}\\ 
&=  c_0^j  \left[ \omega(r) \right]^{-(2-1/\kappa)} A_j^{2-1/\kappa}\,,
\end{align*}
where the constant $c_0$ depends only on $n,p,\Lambda$ and $\kappa$. The lemma on the fast convergence of sequences asserts that $A_j \rightarrow 0$ if
$$A_0 \leq c_0^{-(1-1/\kappa)^{-2}}  \left[ \omega(r) \right]^{\frac{2\kappa-1}{\kappa-1}}\,,$$
which is exactly our assumption \eqref{alt.2}, once we fix the value of $\varepsilon_1$ as 
\begin{equation*}
\varepsilon_1:=\min\big\{c_0^{-(1-1/\kappa)^{-2}},(c_1/16)^{1/(p-2)}\big\}\,. 
\end{equation*}
We conclude that
\begin{equation} \label{friday}
v \leq \osc v - \frac{\omega (r)}{4} \qquad \textrm{in} \ \ B_{r/8} \times \big( \widetilde T_r^\om/2, \widetilde T_r^\om\big)\,.
\end{equation}
Note that $\eps_1$ is a quantity depending only on $n,p,\Lambda$ and $\kappa$ through the dependencies of $c_0$ and $c_1$. This,  via~\eqref{eq:M fix}, fixes also the value of $M$ as a constant depending only on
$n,p,\Lambda$ and possibly on $\kappa$. 

We next need to forward this information in time, and to do this we  first use the logarithmic Lemma \ref{lemma: log lemma} and then another De Giorgi 
iteration. Note, indeed, that now $M\equiv M(n,p,\Lambda,\kappa)$ is fixed; hence, for $\nu^\ast\in(0,1)$ to be chosen, \eqref{friday} together with Lemma \ref{lemma: log lemma} yields
\[
\frac{\Bigl|\big( B_{r/16}\times( \widetilde T_r^\om/2, T_r^\om)\big) \cap \left\{ v \geq \osc v - \varsigma [\omega (r)]^{2+\frac{\kappa}{\kappa-1}} \right\} \Bigr| }{\big| B_{r/16}
\times ( \widetilde T_r^\om/2, T_r^\om)\big| } \leq  \nu^\ast,
\]
for a constant $\varsigma \equiv\varsigma(n,p,\Lambda,\kappa,\nu^*)\in (0,1)$; this will be the starting point of our second iteration. 
 Let indeed
$$V_j := B_{(1+2^{-j}) r/32} \times \big( \widetilde T_r^\om, T_r^\om\big),
\qquad \quad 
B_j := B_{(1+2^{-j}) r/32}\,,
$$
and consider smooth cut-off functions $\phi_j$, depending only on the spatial variables, such that
$$\phi_j \equiv 1 \quad \textrm{in} \ B_{j+1} \quad \text{and} \quad \phi_j = 0 \quad \text{on} \ \partial B_j,\quad\text{with}\quad  \left| D \phi_j \right| \leq \frac{c\,2^j}{r}\,.$$
If we choose a level such that $k \geq \osc v - \omega (r)/4$, then
\begin{equation}\label{zero}
 (v-k)_+ \phi^p  = 0 \qquad \ \text{on}\ \partial_p V_j 
\end{equation}
by \eqref{friday}, so recalling that $1/\alpha=1+\kappa/(\kappa-1)$, we put
$$k_j = \osc v - \frac{(1+2^{-j})}{8} \varsigma\left[ \omega (r) \right]^{2+\frac{\kappa}{\kappa-1}}= \osc v - \frac{(1+2^{-j})}{8} \varsigma\left[ \omega (r) \right]^{\frac{\alpha+1}{\alpha}}\,;$$
note that $k_j\geq \osc v - \frac{\omega (r)}{4}$. We redefine
$$A_j := \frac{\left| V_j \cap \left\{ v \geq k_j \right\}  \right| }{\left| V_j \right|}$$
and observe that
\[
(v-k_j)_+ \leq \varsigma \left[ \omega (r) \right]^{\frac{\alpha+1}{\alpha}}\quad\text{and}\quad(v-k_j)_+ \geq 2^{-j-4} \varsigma\left[ \omega (r) \right]^{\frac{\alpha+1}{\alpha}} \chi_{\{ v \geq k_{j+1}\}}\,.
\]
Using again Caccioppoli's estimate,
\[
\left[ 2^{-j-4} \varsigma \left[ \omega (r) \right]^{\frac{\alpha+1}{\alpha}} \right]^{p + \frac{2\alpha}{1-\alpha}} A_{j+1} \leq c\, r^p \big[T_r^\om\big]^{\frac{\alpha}{1-\alpha}} \biggl[2^{jp}\frac{ \big[ \varsigma \big[ \omega (r) \big]^{\frac{\alpha+1}{\alpha}} \big]^p }{r^p} \biggr]^{\frac{1}{1-\alpha}} A_j^{\frac{1}{1-\alpha}}
\]
because of \eqref{zero} and the fact that $\phi$ is time independent. This implies
\begin{align*}
A_{j+1} & \leq  c^jM^{\frac{\alpha}{1-\alpha}} r^{\frac{p}{1-\alpha}} \varsigma^{\frac{p}{1-\alpha}-p-\frac{2\alpha}{1-\alpha}} \\
&\qquad\times\frac{  [\omega(r)]^{(2-p)(\frac{\alpha+1}{\alpha})\frac{\alpha}{1-\alpha}+(\frac{\alpha+1}{\alpha})[\frac{p}{1-\alpha}-(p + \frac{2\alpha}{1-\alpha})]} \,
}{r^{\frac{p}{1-\alpha}}} A_j^{\frac{1}{1-\alpha}} \\
& =  c^j M^{\frac{\alpha}{1-\alpha}} \varsigma^{\frac{\alpha}{1-\alpha}(p-2)} A_j^\frac{1}{1-\alpha}\\
& \leq  \tilde c^j M^{\frac{\alpha}{1-\alpha}} A_j^{\frac{1}{1-\alpha}}\,,
\end{align*}
since $\varsigma < 1$, and for $\tilde c$ depending on $n,p,\Lambda$ and $\kappa$; recall indeed again that $M\equiv M(n,p,\Lambda,\kappa)$ has already been fixed.
The sequence $A_j$ is then infinitesimal if
\[
A_0 \leq \tilde c^{-(\frac{1-\alpha}{\alpha})^2}  M^{-1} =: \nu^\ast\,;
\]
this fixes the value of $\varsigma$ and also in this case we can conclude
\begin{multline}\label{conclusion 2}
\text{\eqref{ass.trivial} and \eqref{alt.2}}\qquad\Longrightarrow\\
\osc_{B_{r/32} \times ( \widetilde T_r^\om, T_r^\om)} v=\sup_{B_{r/32} \times ( \widetilde T_r^\om, T_r^\om)} v \leq \osc_{Q_r^\om} v - \theta_2\left[ \omega (r) \right]^{2+\frac{\kappa}{\kappa-1}}\,,
\end{multline}
if we call $\theta_2\equiv \theta_2(n,p,\Lambda,\kappa):=\varsigma/8\in(0,1)$; recall that $\widetilde T_r^\om\leq T_r^\om/2$. We have succeeded yet again to reduce the oscillation.

\section{Deriving the modulus of continuity}\label{deriving}

We now show how the results of the previous Section lead to Theorem \ref{interior.continuity}; we fix here the value of $L$ as follows:
\begin{equation}\label{Elle}
L:=\max\Big\{\Big(\frac{32\alpha \ln 32 }{\theta} \Big)^\alpha,2p^\alpha\Lambda\Big\}\,, 
\end{equation}
for $\alpha$ defined in \eqref{ratio.kappa} and $\theta:= \min \{ \theta_1 , \theta_2\}$ (see \eqref{conclusion 1} and \eqref{conclusion 2}), and we consider a cylinder  $\overline Q_{r_0}^{\omega(\cdot)}\subset\Omega_T$, where here is
\begin{equation}\label{starting.cylinder}
\overline Q_r^{\omega(\cdot)} :=B_{r}(x_0)\times (t_0-2^{2-p}\max\{\osc_{\Omega_T} u,1 \}^{2-p}T^\om_r,t_0)\,;
\end{equation}
$T^\om_r=M[ \omega(r) ]^{(2-p)(1+1/\alpha)} r^p$ with $M$ being fixed in \eqref{eq:M fix} and $\omega(\cdot)$ now is defined according to the choice of $L$ performed above. We stress that this in particular gives
\begin{equation}\label{choice modulus.explicit}
\omega (r) \geq \Big(\frac{32\alpha \ln 32 }{\theta} \Big)^\alpha\Bigl[p+\ln\Big(\frac{r_0}{r}\Big)\Bigr]^{-\alpha}\,.
\end{equation}
The scaling we perform now is the one described in  subsection \ref{scaling}, with $T_0=t_0-T^\om_r$, $\overline T=T^\om_r$ and $\lambda:= \max\{\osc_{\Omega_T} u , 1 \}$, which allows to obtain solutions $\bar{u}_\eps$ in 
\[
\hat Q_{r_0}=B_{r_0}\times (t_0-T_{r_0}^{\omega(\cdot)},t_0)\,;
\]
note that 
\[
\osc_{\hat Q_{r_0}} \bar{u}_\eps\leq \frac{1}{2\max\{\osc_{\Omega_T} u , 1 \}}\osc_{\overline{Q}_{r_0}^{\omega(\cdot)}} u_\eps\leq 1
\]
for $\eps>0$ small enough, by local uniform convergence. Note also  that $\eps$ could depend on the starting cylinder in \eqref{starting.cylinder}, but this is not a problem here. What we prove now is
\begin{equation}\label{recover epsilon}
\osc _{\hat Q_r} \bar u_\eps \leq c\, \omega(r)+2^{8}\Lambda\eps\qquad\text{for all $r\leq r_0$}\,,
\end{equation}
for a constant $c$ depending only on $n,p,\Lambda$ and $\alpha$, and this will imply Theorem \ref{interior.continuity} in a straightforward manner, taking into account the assumed local uniform convergence of $u_\eps$ to $u$, scaling back to $\overline Q_r^\om$ and redefining the constant $M$. For radii $r\leq r_0$ we shall consider $w=\beta(\bar{u}_\eps)$ as in \eqref{w}; observe that by the Lipschitz regularity of $\beta$ we have
\[
\osc_{\hat Q_r} w=\osc_{\hat Q_{r_0}} \beta(\bar{u}_\eps)\leq \Lambda\osc_{\hat Q_{r_0}} \bar{u}_\eps\leq \Lambda\,.
\]
Finally, we shall also translate our solution $w$ to $v$ as in \eqref{v}; notice that also $\osc_{\hat Q_r} v\leq \Lambda$. 

\subsection{Iteration} To obtain \eqref{recover epsilon}, we first choose the starting point of our iteration in the following way: noting that $\omega(r_0)\geq\Lambda$, $\omega(\varrho)\to0$ as $\varrho\downarrow0$ and $\omega(\cdot)$ is continuous and increasing, we take the largest (and unique) radius $\tilde r_0\in(0,r_0]$ such that $\omega(\tilde r_0)=\Lambda$. The radius $\tilde r_0$ can be written as $r_0/\tilde c$, where $\tilde c$ depends only on $n,p,\Lambda$ and $\alpha$. We let, for $i\in \N_0$,
\[
r_i:=32^{-i} \tilde r_0,\qquad\text{and}\qquad Q_i := \hat Q_{r_i}=B_{r_i} \times \bigl(t_0-T_{r_i}^\om,t_0\bigr)\,;
\]
from now on, we will work with the function $v$ defined just above. From the analysis of Section \ref{reducing}, we got that if $\omega (r_i) \leq \osc_{Q_i} v$ and $\varepsilon< \omega (r_i)/8$, then
\begin{equation} \label{eq:basic conclusion from alts}
\osc_{Q_{i+1}}  v \leq \osc_{Q_i}  v - \theta \left[\omega (r_i)\right]^{2+\frac{\kappa}{\kappa-1}}\,.
\end{equation}
Indeed, following subsection \ref{scaling}, rescale $v$ defined in $Q_i$ to $\bar v$ in $B_{r_i}\times(0,T_{r_i}^{\omega(\cdot)})$ (take $\lambda=1$); since $\omega (r_i) \leq \osc_{B_{r_i}\times(0,T_{r_i}^{\omega(\cdot)})} \bar v$, \eqref{conclusion 1} and \eqref{conclusion 2} give that
\[
\osc_{B_{r_i/32} \times (\frac{3}{4}T_{r_i}^\om, T_{r_i}^\om)}\bar v \leq \osc_{B_{r_i}\times(0,T_{r_i}^{\omega(\cdot)})}\bar  v - \theta\left[\omega (r_i)\right]^{2+\frac{\kappa}{\kappa-1}}
\]
and, after scaling back, \eqref{eq:basic conclusion from alts} is a consequence of the fact that $T_{r_{i+1}}^\om\leq\frac14T_{r_i}^\om$: indeed, a direct calculation shows that
\begin{equation}\label{doubling.property}
\frac{\omega'(\varrho)\varrho}{\omega(\varrho)}\leq \frac{\alpha}{p}\ \ \text{for $0<\varrho\leq r_0$}\quad\Longrightarrow\quad \frac{\omega(\varrho_2)}{\omega(\varrho_1)}\leq \Bigl(\frac{\varrho_2}{\varrho_1}\Bigr)^{\frac\alpha p}\ \ \text{for }\varrho_1\leq \varrho_2\leq \varrho_0\,.
\end{equation}

\vspace{3mm}

We now show that if $\eps<\omega(r_{\bar\imath})/8$ for some $\bar\imath\in\mathbb N$, then
\begin{equation}\label{ind.step}
\osc_{Q_i}  v \leq 32 \, \omega (r_i) 
\end{equation}
for all $i\in\{0,1,\dots,\bar\imath+1\}$. Suppose then that \eqref{ind.step} holds for $i\in\{0,1,\dots,j\}$, with $j\leq \bar\imath$ and let's prove that it holds for $j+1$; note that, by the monotonicity of $\omega$, we have $\eps<\omega(r_i)/8$ for $i\in\{0,1,\dots,j\}$. 
Let now $i^\ast$ be the largest integer in $\{0,1,\ldots,j\}$ such that $\osc_{Q_{i^\ast}} v< \omega (r_{i^\ast})$ holds; note that such an index exists since $\osc_{Q_1} v \leq \Lambda = \omega(\tilde r_0)$ by our choice of $\tilde r_0$, and moreover this fixes the inductive starting step. If $i^\ast = j$, then the induction step follows from the doubling property of $\omega$, i.e., $\omega(r_{j}) \leq 32\omega(r_{j+1})$. Assume then that $i^\ast < j$ so that, by the induction assumption, we have
\[
\omega (r_i) \leq \osc_{Q_{i}} v \leq  32\omega (r_{i}) , \qquad \forall \ i \in \{i^\ast+1,\ldots,j\}\,.
\]
Therefore, \eqref{eq:basic conclusion from alts} is at our disposal for any such index (recall $\eps<\omega(r_i)/8$ for all $i\leq j$) and it leads to
$$\osc_{Q_{i+1}} v \leq \osc_{Q_{i}} v - \theta \, [\omega (r_i)]^{2+\frac{\kappa}{\kappa-1}} \leq \left( 1-\frac{\theta}{32} [\omega (r_i)]^{1+\frac{\kappa}{\kappa-1}}\right) 
\osc_{Q_{i}} v$$
for $i \in \{i^\ast+1,\ldots,j\}$. Iterating and using also the fact that $\osc_{Q_{i^\ast+1}} v \leq \osc_{Q_{i^\ast}} v \leq \omega (r_{i^\ast})$, we get
\begin{equation} \label{iter}
\osc_{Q_{j+1}} v \leq \prod_{i=i^\ast+1}^j\left( 1-\frac{\theta}{32} [\omega (r_i)]^{1+\frac{\kappa}{\kappa-1}}\right) \omega (r_{i^\ast})\,.
\end{equation}
Now, recalling that $1/\alpha=1+\kappa/(\kappa-1)$ and using \eqref{choice modulus.explicit}, we estimate
\begin{eqnarray*}
\prod_{i=i^\ast+1}^j\left( 1-\frac{\theta}{32} [\omega (r_i)]^{1+\frac{\kappa}{\kappa-1}}\right) &=& \exp \left( \sum_{i=i^\ast+1}^j \ln \left( 1-\frac{\theta}{32}[\omega (r_i)]^{1+\frac{\kappa}{\kappa-1}}\right)\right)\\
& \leq&   \exp\biggl(-\frac{\theta}{32} \frac1{\ln 32}\int_{r_j}^{r_{i^\ast}}  [\omega (\rho)]^{\frac1\alpha} \frac{d \rho}{\rho}\biggr)\\
& = &  \exp \left( -\alpha\int_{r_j}^{r_{i^\ast}}  \frac{1}{p+\ln \big( \frac{r_0}{\rho}\big)} \frac{d \rho}{\rho}\right)\\
& = &  \exp \left( -\alpha \left[ \ln \ln \Bigl(\frac{e^p r_0}{r_j} \Bigr)-  \ln \ln\Bigl( \frac{e^p r_0}{r_{i^\ast}} \Bigr)\right] \right)\\
& = &  \exp \left( - \ln \left[ \frac{p+\ln \big( \frac{r_0}{r_j}\big)}{p+\ln \big( \frac{r_0}{r_{i^\ast}}\big)}
\right]^\alpha \right)=  \frac{\omega (r_j)}{\omega (r_{i^\ast})}\,.
\end{eqnarray*}
Indeed, from the elementary estimate $\ln (1-x) \leq -x$ if $x < 1$ and the fact that $\theta[\omega (r_i)]^{1+\frac{\kappa}{\kappa-1}}/32<1$, we have
\begin{align*}
\sum_{i=i^\ast+1}^j\ln \left( 1-\frac{\theta}{32}[\omega (r_i)]^{1+\frac{\kappa}{\kappa-1}}\right) & \leq     -\frac{\theta}{32}\sum_{i=i^\ast+1}^j   [\omega (r_i)]^{1+\frac{\kappa}{\kappa-1}}\\
& \leq   -\frac{\theta}{32} \frac1{\ln 32}\int_{r_j}^{r_{i^\ast}}  [\omega (\rho)]^{\frac1\alpha} \frac{d \rho}{\rho}\,,
\end{align*}
using also the fact that $\omega (\cdot)$ is increasing. Inserting this computation in \eqref{iter} and using the doubling property of $\omega (\cdot)$, we get
\[
\osc_{Q_{j+1}} v  \leq   \frac{\omega (r_j)}{\omega (r_{i^\ast})} \omega (r_{i^\ast})  = \omega (r_j)  \leq 32 \, \omega (r_{j+1})
\]
and the (finite) induction is complete.

\subsection{Conclusion} To conclude, for $\eps\in(0,1)$ fixed, corresponding to the solution $v$, see \eqref{w} and \eqref{v}, take $\bar\imath\in\mathbb N$ as the smallest index such that $\omega(r_{\bar\imath})/8\geq\eps$. By \eqref{ind.step}, we have
\[
\osc_{Q_i} v \leq 32 \, \omega (r_i), \qquad\text{for $i\in\{0,1,\dots,\bar\imath\}$}\,.
\]
Now for a radius $r\in(r_{\bar\imath+1},\tilde r_0]$, call $\hat\imath$ the index in $\{0,1,\dots,\bar\imath\}$ such that $r_{\hat\imath+1}<r\leq r_{\hat\imath}$. We have
\[
\osc_{\hat Q_r} v\leq \osc_{Q_{\hat\imath}} v \leq 32 \, \omega (r_{\hat\imath})\leq (32)^2\,\omega (r_{\hat\imath+1})\leq (32)^2\,\omega (r)\,;
\]
on the other hand, for $r\in(0,r_{\bar\imath+1}]$ we trivially estimate 
$$
\osc_{\hat Q_r} v\leq \osc_{Q_{\bar\imath+1}} v\leq 32\omega(r_{\bar\imath+1})<2^8\eps\,. 
$$
Finally, if $r\in(\tilde r_0,r_0]$, we simply use \eqref{doubling.property} in the following way:
\[
\osc_{\hat Q_r} v\leq \omega(r_0)\leq \omega(\tilde r_0) \Bigl(\frac{r_0}{\tilde r_0}\Bigr)^{\frac\alpha p}\leq c\,\omega(\tilde r_0)\leq c\,\omega(r)\,,
\]
recalling that $\tilde r_0\equiv r_0/\tilde c(n,p,\Lambda,\alpha)$. \eqref{recover epsilon} now follows recalling that $v$ is a translation of $w$ and taking into account the Lipschitz property of $\beta$.

\vspace{5mm}

{\small {\bf Acknowledgments:} This paper was conceived, and partially written, while PB and JMU were visiting Aalto University and when TK visited the University of Coimbra, while it was finalised during the program ``Evolutionary Problems'' at the Institut Mittag-Leffler (Djursholm, Sweden) in the Fall 2013. The authors gratefully acknowledge the support and the hospitality of all these institutions. 

Research of JMU partially supported by FCT projects UTA-CMU/MAT/0007/2009 and PTDC/MAT-CAL/0749/2012, by FCT grant SFRH/BSAB/1273/2012, and by the CMUC, funded by the European Regional Development Fund, through the program COMPETE, and by the Portuguese Government, through the FCT, under the project PEst-C/MAT/UI0324/2011. TK was supported by the Academy of Finland project ``Regularity theory for nonlinear parabolic partial differential equations". Research of PB and TK were partially supported by the ERC grant 207573 ``Vectorial Problems''. }


\begin{thebibliography}{99}

\bibitem{AM07}{\sc E.~Acerbi and G.~Mingione},
Gradient estimates for a class of parabolic systems,
{\em Duke Math. J.} {\bf 136~(2)} (2007), 285--320.

\bibitem{ACS1}
{\sc I. Athanasopoulos, L. Caffarelli and S. Salsa}, 
Caloric functions in Lipschitz domains and the regularity of solutions to phase transition problems, 
{\em Ann. Math.} {\bf 143} (1996), no. 3, 413--434.

%

\bibitem{BKU2}
{\sc  P. Baroni, T. Kuusi and J.M. Urbano},
On the boundary regularity in phase transition problems,
{\em in preparation}.

%

\bibitem{CafEv83}
{\sc L. Caffarelli and L.C. Evans},
Continuity of the temperature in the two-phase Stefan problem,
{\em Arch. Rational Mech. Anal.} {\bf 81} (1983), no. 3, 199--220.

\bibitem{CaFri79}
{\sc L. Caffarelli and A. Friedman},
Continuity of the temperature in the Stefan problem,
{\em Indiana Univ. Math. J.} {\bf 28} (1979), no. 1, 53--70.


\bibitem{Bio1}
{\sc E. N. Dancer, D. Hilhorst, M. Mimura and L.A. Peletier},
Spatial segregation limit of a competition-diffusion system,
{\em European J. Appl. Math.} {\bf 10} (1999),  97--115.

\bibitem{DiB2}
{\sc E. DiBenedetto},
Continuity of weak solutions to certain singular parabolic equations,
{\em Ann. Mat. Pura Appl. {\rm (4)}} {\bf 103} (1982), 131--176.

\bibitem{DiB3}
{\sc E. DiBenedetto},
A boundary modulus of continuity for a class of singular parabolic equations,
{\em J. Differential Equations} {\bf 63} (1986), no. 3, 418--447.


\bibitem{DiBe93}
{\sc E. DiBenedetto},
{\em Degenerate parabolic equations}, 
Universitext, Springer-Verlag, New York, 1993.


\bibitem{DiBeFrie84}
{\sc E. DiBenedetto and A. Friedman},
Regularity of solutions of nonlinear degenerate parabolic systems,
{\em J. reine angew. Math. {\rm(}Crelle's J.{\rm)}} {\bf 349} (1984), 83--128.


\bibitem{DiBeGianVesp08}
{\sc E. DiBenedetto, U. Gianazza and V. Vespri},
Harnack estimates for quasi-linear degenerate parabolic differential equations,
{\em Acta Math.} {\bf 200} (2008), no. 2, 181--209.

\bibitem{DiBeGianVesp12}
{\sc E. DiBenedetto, U. Gianazza and V. Vespri},
{\em Harnack's inequality for degenerate and singular parabolic equations}, 
Springer Monographs in Mathematics. Springer, New York, 2012.


\bibitem{DiBeVesp95}
{\sc E. DiBenedetto and V. Vespri}, On the singular equation $\beta(u)_t=\triangle u$. {\em Arch. Rational Mech. Anal.} 132(3) (1995), 247--309.

\bibitem{Friedman}
{\sc A. Friedman},
The Stefan problem in several space variables, 
{\em Trans. Amer. Math. Soc.} {\bf 133} (1968), 51--87.

\bibitem{Friedman_book}
{\sc A. Friedman},
{\em Variational principles and free-boundary problems},
John Wiley and Sons, New York, 1982.

\bibitem{GSV}
{\sc U. Gianazza, M. Surnachev and V. Vespri}, 
A new proof of the H\"older continuity of solutions to $p$-Laplace type parabolic equations,
{\em Adv. Calc. Var.} {\bf 3} (2010), no. 3, 263--278.

\bibitem{NoUr03}
{\sc N. Igbida and J.M. Urbano},
Uniqueness for nonlinear degenerate problems,
{\em NoDEA Nonlinear Differential Equations Appl.} {\bf 10} (2003), no. 3, 287--307.

\bibitem{Kam}
{\sc S.L. Kamenomostskaya}, 
On Stefan problem, 
{\em Mat. Sbornik} {\bf 53} (1961), 489--514 (Russian).

\bibitem{Kim}
{\sc I.C. Kim and N. Pozar}, 
Viscosity solutions for the two phase Stefan Problem,
{\em Comm. Partial Differential Equations} {\bf 36} (2011), 42--66.


\bibitem{Kuu08}
{\sc T. Kuusi},
Harnack estimates for weak supersolutions to nonlinear degenerate parabolic equations,
{\em Ann. Sc. Norm. Super. Pisa Cl. Sci. {\rm (5)}} {\bf 7} (2008), no. 4, 673--716.

\bibitem{KMN}
{\sc T. Kuusi, G. Mingione and K. Nystrom},
Sharp regularity for evolutionary obstacle problems, interpolative geometries and removable sets,
to appear in {\em J. Math. Pures Appl.}, doi: 10.1016/j.matpur.2013.03.004

\bibitem{russ}
{\sc O.A.~Ladyzhenskaya, V.A.~Solonnikov and N.N.~Ural'tseva},
{\em Linear and quasi-linear equations of parabolic type},
Transl. Math. Monographs, vol. 23, Amer. Math. Soc., 1968.


\bibitem{Noche87}
{\sc R.H. Nochetto},
A class of nondegenerate two-phase Stefan problems in several space variables,
{\em Comm. Partial Differential Equations} {\bf 12} (1987), no. 1, 21--45.

\bibitem{Ole}
{\sc O.A. Oleinik}, 
A method of solution of the general Stefan problem,
{\em Soviet Math. Dokl.} {\bf 1} (1960), 1350--1354.

\bibitem{finance_book}
{\sc G. Peskir and A. Shiryaev},
{\em Optimal stopping and free-boundary problems},
Lectures in Mathematics ETH Z\"urich, Birkh\"auser Verlag, Basel, 2006.

\bibitem{Rubin}
{\sc J. Rubin},
Transport of reacting solutes in porous media: relation between mathematical nature of problem formulation and chemical nature of reactions,
{\em Water resources research} {\bf 19} (1983), no. 5, 1231--1252.

\bibitem{Sacks}
{\sc P.E. Sacks}, 
Continuity of solutions of a singular parabolic equation,
{\em Nonlinear Anal.} {\bf 7} (1983), no. 4, 387--409. 

\bibitem{Salsa12}
{\sc S. Salsa},
Two-phase Stefan problem. Recent results and open questions.,
{\em Milan J. Math.} {\bf 80} (2012), no. 2, 267--281.

\bibitem{Ste_before_history}
{\sc J. Stefan},
\"Uber die Theorie der Eisbildung,
{\em Monatshefte Mat. Phys.} {\bf 1} (1890), 1--6.
	 
\bibitem{UEx}
{\sc J.M. Urbano},
A free boundary problem with convection for the $p$-Laplacian,
{\em Rend. Mat. Appl. {\rm (7)}} {\bf 17} (1997), no. 1, 1--19.

\bibitem{UCon}
{\sc J.M. Urbano},
Continuous solutions for a degenerate free boundary problem,
{\em Ann. Mat. Pura Appl. {\rm (4)}} {\bf 178} (2000), 195--224. 

\bibitem{Urba08}
{\sc J.M. Urbano},
{\em The method of intrinsic scaling, A systematic approach to regularity for degenerate and singular PDEs},
Lecture Notes in Mathematics, 1930, Springer-Verlag, Berlin, 2008.

\bibitem{Ziemer}
{\sc W.P. Ziemer},
Interior and boundary continuity of weak solutions of degenerate parabolic equations,
{\em Trans. Amer. Math. Soc.} {\bf 271} (1982), no. 2, 733--748. 
\end{thebibliography}
\end{document}